\documentclass[preprint,12pt]{elsarticle}

\usepackage{amssymb}
\usepackage[latin1]{inputenc}
\usepackage[english]{babel}
\usepackage{hyperref}
\usepackage{amsthm}
\usepackage{amsmath}
\usepackage{amsxtra}
\usepackage{mathrsfs}
\usepackage{amssymb}

\input xy
\xyoption{all}

\newcommand{\tup}[1]{\left\langle#1\right\rangle}
\DeclareMathOperator{\acc}{acc}
\DeclareMathOperator{\cf}{cf}
\DeclareMathOperator{\dom}{dom}

\DeclareMathOperator{\stem}{stem}

\newcommand{\Coll}{{\rm Coll}}

\newcommand{\cof}{{\rm cof}}
\newcommand{\ZFC}{{\rm ZFC }}

\newcommand{\force}{\Vdash}

\newcommand{\restr}{\upharpoonright}

\newtheorem{theorem}{Theorem}[section]
\newtheorem{lemma}[theorem]{Lemma}
\newtheorem{proposition}[theorem]{Proposition}
\newtheorem{coroll}[theorem]{Corollary}
\newtheorem{remark}[theorem]{Remark}
\newtheorem{definition}[theorem]{Definition}
\newtheorem{question}[theorem]{Question}
\newtheorem{notation}[theorem]{Notation}

\journal{Annals of Pure and Applied Logic}

\begin{document}

\begin{frontmatter}



\title{Square and Delta reflection}


\author[HUJ]{Laura Fontanella}
\ead{laura.fontanella@mail.huji.ac.il}
\author[HUJ]{Yair Hayut} 
\ead{yair.hayut@mail.huji.ac.il}
\address[HUJ]{Hebrew University of Jerusalem\\
Einstein Institute of Mathematics, Edmond J. Safra campus\\
Givat Ram, 91904 Jerusalem, Israel}
%\curraddr{}

\begin{abstract}

Starting from infinitely many supercompact cardinals, we force a model of ZFC where $\aleph_{\omega^2+1}$ satisfies simultaneously a strong principle of reflection, called $\Delta$-reflection, and a version of the square principle, denoted $\square(\aleph_{\omega^2+1}).$ Thus we show that $\aleph_{\omega^2+1}$ can satisfy simultaneously a strong reflection principle and an anti-reflection principle.

\end{abstract}

\begin{keyword}
reflection principles, square, large cardinals, forcing.



\MSC[2010]{Primary: 03E55. Secondary: 03E35, 03E65} 

%\subjclass[2010]{Primary: 03E55. Secondary: 03E35, 03E65}


\end{keyword}

\end{frontmatter}



\section{Introduction}

One of the most fruitful research areas in set theory is the study of the so-called \emph{reflection principles}. A reflection principle is, roughly, a statement establishing that for a certain class of structures, if the structure satisfies a given property, then there is a substructure of smaller cardinality that satisfies the same property. Reflection principles can be seen as versions of Downward Lowenheim-Skolem theorem and they derive from large cardinal notions such as the notion of strongly compact cardinal.  Magidor and Shelah introduced in 1994 (see \cite{MagidorShelah}) a two-cardinals reflection principle, denoted $\Delta_{\lambda,\kappa}$ where $\lambda<\kappa$ (see the definition in Section~\ref{sec: prelim}). This is a strong reflection principle as it implies many interesting properties of structures of various kind. For instance, given a cardinal $\kappa,$ if $\Delta_{\omega_1, \kappa}$ holds, then every almost free Abelian group of size $\kappa$ is free; if $\Delta_{\lambda,\kappa}$ holds for every $\lambda<\kappa$ -- we say that $\kappa$ has the \emph{Delta reflection} -- then for every graph $G$ of size $\kappa,$ if every subgraph of smaller size has coloring number $\mu<\kappa,$ then $G$ itself has coloring number $\mu;$ similar properties hold for other kind of structures under the assumption that $\kappa$ has the Delta reflection (see \cite{Shelah} and \cite{MagidorShelah}). It should be pointed out that these statements are always true when $\kappa$ is a singular cardinal (see \cite{Shelah}), thus we are interested in those \emph{regular} cardinals that satisfy the Delta reflection. 

The Delta reflection implies another classical reflection principle, namely the \emph{stationary set reflection} which for a cardinal $\kappa$ established that for every stationary subset $S$ of $\kappa$ of points of cofinality less than $\lambda,$ there exists $\alpha<\kappa$ such that $S\cap \alpha$ is stationary in $\alpha.$ Since the stationary set reflection fails at the successor of a regular cardinal, in particular the Delta reflection can only hold at successors of singular cardinals. The Delta reflection follows from weakly compact cardinals, however Magidor and Shelah showed in \cite{MagidorShelah} that it is consistent that a small cardinal, namely $\aleph_{\omega^2+1},$ has the Delta reflection, assuming \ZFC is consistent with the existence of infinitely many supercompact cardinals. Moreover, the results proven in \cite{MagidorShelah}, combined with other older results by Eklof \cite{Eklof}, Milner and Shelah \cite{MilnerShelah}, \cite{Shelah}, imply that $\aleph_{\omega^2+1}$ is the smallest regular cardinal that can have the Delta reflection. Recent work by the first author and Magidor \cite{FontanellaMagidor} shows that $\aleph_{\omega^2+1}$ can even satisfy simultaneously the Delta reflection and the tree property, which is another strong reflection principle (more precisely, assuming \ZFC is consistent with the existence of infinitely many supercompact cardinals, then \ZFC is consistent with $\aleph_{\omega^2+1}$ having both the Delta reflection and the tree property); on the other hand, it is shown in \cite{FontanellaMagidor} that 
the Delta reflection at $\aleph_{\omega^2+1}$ does not imply the tree property at this cardinal (nor the tree property implies the Delta reflection), thus the Delta reflection and the tree property are independent.         

 In this paper we show that, assuming the consistency of infinitely many supercompact cardinals, it is possible to force a model of \ZFC where $\aleph_{\omega^2+1}$ satisfies simultaneously the Delta reflection and a version of the square, denoted $\square(\aleph_{\omega^2 +1})$ (see the definition in Section~\ref{sec: prelim}). Introduced by Todor\v{c}evi\'c in \cite{Todorcevic}, $\square(\kappa)$ is an anti-reflection principle on a cardinal $\kappa$. For example, as demonstrated in \cite{Todorcevic}, $\square(\kappa)$ implies the failure of the tree property at $\kappa.$ It also implies the failure of the simultaneous reflection, namely every stationary set can be split into two disjoint stationary parts such that there is no ordinal $\alpha < \kappa$ in which they are both stationary (see Veli\v{c}kovi\'c \cite{Boban}). Rinot \cite{Rinot} proved that an even stronger failure of the simultaneous reflection follows from $\square(\kappa),$ namely any stationary subset of $\kappa$ can be partitioned into $\kappa$ many pairwise disjoint stationary sets such that no two of them reflect simultaneously. %This illustrate the non-compactness that $\square(\kappa)$ reveals. 
It follows by a result of Solovay \cite{Solovay} that $\square(\lambda)$ fails if there is a $\lambda$-strongly compact cardinal, so in a way the existence of a $\square(\lambda)$ sequence bounds the amount of downward reflection that we can get on structures of size $\lambda$.  

$\square(\kappa^+)$ is a consequence of the well-known \emph{Jensen's square principle} $\square_\kappa$ (see Jensen \cite{Jensen}). The Delta reflection implies the failure of the weak square $\square^*_\kappa$ which is another weak consequence of $\square_\kappa$ (hence a fortiori the Delta reflection implies the failure of $\square_\kappa$), thus our result implies that one can have at $\aleph_{\omega^2+1}$ a good balance between a reflection principle and an anti-reflection principle.     

\section{Preliminaries and notation}\label{sec: prelim}

In this section we give the definition of the Delta reflection and $\square(\lambda).$ Then we prove some preliminary results about the Delta reflection that will be used in the final proof of our main theorem.  

\begin{notation} Let $\kappa<\mu$ be two regular cardinals, we denote by $E_{<\kappa}^{\mu}$ the set $\{\alpha<\mu\mid\cof(\alpha)<\kappa\}$, and we denote by $E_\kappa^\mu$ the set $\{\alpha < \mu \mid \cof(\alpha)=\kappa\}.$
\end{notation}

\begin{notation} Let $f$ be a function and $A$ be a set, we denote by $f[A]$ the set $\{f(x) \mid x\in A\cap \dom f\}.$ 
\end{notation}

Given a forcing notion $\mathbb{P}$ and two conditions 
$p, q\in \mathbb{P},$ we write $p\leq q$ when $p$ is stronger than $q.$ 
Given a cardinal $\kappa,$ we recall that 
\begin{itemize}
\item $\mathbb{P}$ is \emph{$\kappa$-closed} if every decreasing sequence of less than $\kappa$ many conditions in $\mathbb{P}$ has a lower bound;  
\item $\mathbb{P}$ is \emph{$\kappa$-directed closed} if for every set of less than $\kappa$ many pairwise compatible conditions in $\mathbb{P},$ it is possible to find a lower bound; 
\item $\mathbb{P}$ is \emph{$\kappa$-distributive} if it does not add sequences of ordinals of length less than $\kappa;$
\item $\mathbb{P}$ is \emph{$\kappa$-c.c.} if every antichain has size less than $\kappa;$
\item $\mathbb{P}$ is \emph{$\kappa$-centered} if it can be split into $\kappa$ many pieces that are centered, namely for each piece $X$ every finite subset $Y$ of $X$ has a lower bound. 
\end{itemize}

For a forcing notion $\mathbb{P}$ and a model of \ZFC, $V$, we write $V^\mathbb{P}$ for the Boolean-valued model of $\mathbb{P}$. In particular, we write $V^\mathbb{P}\models \varphi$ iff $1_\mathbb{P}\force \varphi$. 
 
\begin{definition}\label{def: Delta} Given two cardinals $\kappa<\mu,$ $\Delta_{\kappa, \mu}$ is the statement that for every cardinal $\nu<\kappa,$ for every stationary set $S\subseteq E_{<\kappa}^{\mu}$ and for every algebra $A$ on $\mu$ with $\nu$ many operations, there exists a subalgebra $A'$ of order type a regular cardinal below $\kappa$ such that $S\cap A'$ is stationary in $\sup(A').$ 
\end{definition}

We say that $\mu$ has the \emph{$\Delta$-reflection} (or \emph{Delta reflection}) if $\Delta_{\kappa,\mu}$ holds for every $\kappa<\mu.$  
%The Delta-reflection is a powerful principle as it implies interesting properties for structures of different kind. 

%$\Delta_{\kappa, \mu}$ is a powerful reflection principle. It can be seen as a strengthening of the Downward Lowenheim-Skolem theorem as, roughly speaking, it allows us to control not only the cardinality of the sub-algebra but also the way that it meets a given stationary set and its order type. 
%In \cite{MagidorShelah}, Magidor and Shelah formulate a general axiomatic system, and derive many consequences of $\Delta_{\kappa, \mu}$. 
%For example $\Delta_{\kappa, \aleph_1}$ implies that every almost-free %Abelian group of cardinality $\kappa$ is free.   

The proof of the following proposition is based on an argument of Solovay who showed that given a strongly compact cardinal $\kappa,$ if $\mu>\kappa$ is any regular cardinal, then every stationary subset of $E_{<\kappa}^{\mu}$ reflects. 

\begin{proposition}\label{prop: reflection from spc}  Let $\kappa$ be a supercompact cardinal and let $\mu>\kappa.$ 
Suppose $S\subseteq E_{<\kappa}^{\mu}$ is a stationary set and $A$ is an algebra on $\mu$ with less than $\kappa$ many operations, then there exists a set $X$ such that if $B$ is the subalgebra of $A$ generated by $X,$ then:
\begin{itemize}
\item $o.t.(B)= o.t.(X)<\kappa$ is a regular cardinal
\item $S\cap B$ is stationary in $\sup(B).$   
\end{itemize}
Moreover, if $\mu= \kappa^{+\omega+1},$ then $o.t.(B)= (X\cap\kappa)^{+\omega+1}.$   \end{proposition}

\begin{proof} We can suppose that for some $\chi<\kappa,$ we have $S\subseteq E_\chi^{\mu}.$ Let $j: V\to M$ be a $\mu$-supercompact embedding with critical point $\kappa.$ Suppose $\langle f_i:\ i<\nu\rangle$ are the operations of the algebra $A$ with $\nu<\kappa.$ Then $j(S)$ is a stationary subset of $E_{\chi}^{j(\mu)}$ and $j(A)$ is an algebra on $j(\mu)$ with operations 
$\langle j(f_i):\ i<\nu\rangle.$ We consider the subalgebra $B^*$ of $j(A)$ generated by $j[\mu],$ note that by the closure of $M$ we have $B^*\in M.$ The domain of $B^*$ is precisely $j[\mu]$ for if $j(f_i)$ is $n$-ary and $j(\alpha_1), \ldots, j(\alpha_n)$ are ordinals in $j[\mu],$ then 
$j(f_i)(j(\alpha_1), \ldots, j(\alpha_n))= j(f_i(\alpha_1,\ldots, \alpha_n))= j(\eta)$ for some $\eta\in \mu.$ It follows that the order type of $B^*$ is $\mu<j(\kappa).$
We claim that $j(S)\cap B^*$ is stationary in $\sigma^*:=\sup j[\mu]$ indeed, if $C^*\subseteq \sigma^*$ is a club, then $C:= j^{-1}[C^*]$ is a $<\kappa$-closed unbounded subset of $\mu.$ Since $S$ is a stationary subset of $E_{<\kappa}^{\mu}$, we can find $\zeta\in C\cap S.$ It follows that $j(\zeta)\in j(S)\cap C^*$ thus $j(S)\cap B^*$ has non empty intersection with every club of $\sigma^*.$ 

We proved 
$$\begin{array}{rl}
M\models & \exists X \textrm{ of order type a regular cardinal $<j(\kappa)$ such that}\\
& \textrm{$X$ generates a subalgebra $B$ of $j(A)$ such that $o.t.(B)= o.t.(X)$}\\
& \textrm{and $j(S)\cap B$ is stationary in $\sup(B).$} 
\end{array}
$$

By elementarity, we have 
$$\begin{array}{rl}
V\models & \exists X \textrm{of order type a regular cardinal $<\kappa$ such that}\\
& \textrm{$X$ generates a subalgebra $B$ of $A$ such that $o.t.(B)= o.t.(X)$}\\
& \textrm{and $S\cap B$ is stationary in $\sup(B).$} 
\end{array}
$$

Moreover, if $\mu=\kappa^{+\omega+1}=(\kappa^{+\omega+1})^M,$ then $M$ satisfies $(j[\mu]\cap j(\kappa))= \kappa$ and $o.t.(B^*)= \mu= \kappa^{+\omega+1}.$
It follows, by elementarity, that in $V$ we have a subalgebra $B$ as above such that in addition $o.t.(B)= (X\cap \kappa)^{+\omega+1}.$ That completes the proof of the proposition. 
\end{proof}

\begin{coroll}\label{coroll: Delta reflection at ZFC} Assume $\lambda$ is a singular limit of supercompact cardinals $\langle \kappa_n:\ n<\omega\rangle,$ then $\Delta_{\lambda,\lambda^+}$ holds.  
\end{coroll}

\begin{proof} Let $S$ be a stationary subset of $\lambda^+$ and let $A$ be an algebra on $\lambda^+$ as in the statement of $\Delta_{\lambda,\lambda^+}.$ For some $n<\omega,$ the set $S^*:= \{\alpha\in S;\ cof(\alpha)<\kappa_n\}$ is stationary and $A$ has less than $\kappa_n$ many operations. Using the supercompactness of $\kappa_n,$ we can apply Proposition \ref{prop: reflection from spc}, so $S^*$ reflects on a subalgebra of $A$ of order type a regular cardinal below $\lambda^+$, and since $S^*\subseteq S$, clearly $S$ reflects at this subalgebra as well. \end{proof}

\begin{definition}
Let $\kappa$ be a regular cardinal. 

$\square(\kappa)$ holds if there is a sequence $\mathcal{C} = \tup{C_\alpha \mid \alpha < \kappa}$ such that:
\begin{enumerate}
\item For every $\alpha < \kappa$, $C_\alpha \subseteq \alpha$ is closed and unbounded. If $\alpha$ is a successor ordinal we will assume that $C_\alpha = \{\alpha - 1\}.$
\item For every $\beta \in \acc C_\alpha$, $C_\alpha \cap \beta = C_\beta$. 
\item There is no thread, i.e. there is no club $D \subseteq \kappa$ such that for every $\alpha \in \acc D,$ $D \cap \alpha = C_\alpha$. 
\end{enumerate} 
$\mathcal{C}$ is called a square sequence.
\end{definition}

This definition, due to Todor\v{c}evi\'c \cite{Todorcevic}, is a second order variant of Jensen's square, $\square_\kappa$. It is an anti-reflection principle, as $\kappa$ is the first ordinal in which there is no candidate for $C_\kappa$.
Note that a $\square(\kappa)$-sequence is a sequence of length $\kappa,$ unlike $\square_\kappa$-sequences which are sequences of length $\kappa^+.$

Jensen proved that in $L$ and other core models, $\square(\kappa)$ holds if and only if $\kappa$ is not weakly compact (see \cite{Jensen}). Conversely, if $\kappa$ is a regular cardinal in which $\neg \square(\kappa)$ holds then $\kappa$ is weakly compact in $L$ (see \cite[Prop. 6.1]{ShelahStanley}). Also, by a result of Solovay $\square(\kappa)$ fails for every $\kappa$ above a strongly compact cardinal. 

\section{Definition of the model}\label{sec: definition of the model}

The result we want to prove is the following. 

\begin{theorem}\label{thm: main theorem} Assume that \ZFC is consistent with the existence of infinitely many supercompact cardinals, then \ZFC is consistent with $\Delta_{\aleph_{\omega^2}, \aleph_{\omega^2+1}}$ plus $\square(\aleph_{\omega^2+1}).$
\end{theorem}

In this section we define the forcing construction that will give us a model of 
$\Delta_{\aleph_{\omega^2}, \aleph_{\omega^2+1}}$ plus $\square(\aleph_{\omega^2+1}).$ 

Let $\lambda:= \sup_{n<\omega} \kappa_n$ where $\langle \kappa_n \mid n<\omega \rangle$ is an increasing sequence of supercompact cardinals. Using a construction by Laver \cite{Laver}, we can assume that each $\kappa_n$ is indestructibly supercompact, i.e. if $\mathbb{A}$ is a $\kappa_n$-directed closed forcing, then $V^{\mathbb{A}}\models \kappa_n$ is supercompact. We also assume that $2^{\kappa_n}=\kappa_n^+$ holds for every $n<\omega.$ 

By Corollary~\ref{coroll: Delta reflection at ZFC}, $\Delta_{\lambda^+, \lambda}$ holds. We want to collapse the cardinals below $\lambda^+$ in a way that will give us $\lambda = \aleph_{\omega^2}$, $\lambda^{+} = \aleph_{\omega^2 + 1}$ and still $\Delta_{\lambda^+, \lambda}$ holds. This is done by the construction due to Magidor and Shelah. We want also to have $\square(\lambda^+)$ in this model. This is more problematic, as if $\square(\lambda^+)$ holds in $V$ then there are no $\lambda^+$-supercompact cardinals. The trick, which is quite standard, 
%and originate at [???], 
is to force a generic square sequence at $\lambda^+$, then use the fact that, after we further force a thread for this square sequence, the iteration is $\lambda^+$-directed closed. Therefore we restore the supercompactness of any indestructible $\lambda^+$-supercompact from the ground model. In particular, the threading forcing restores the Delta principle at singular limits of supercompact cardinals. So, in order to get the Delta reflection together with the square principle, we must argue that if, by contradiction, the Delta reflection fails before we force the thread, then the counterexample would be preserved by the threading forcing.  

There are two difficulties when trying to work in this direction. The first problem is the following: although $\Delta_{\lambda^+, \lambda}$ is indestructible under $\lambda^+$-directed closed forcing (because $\lambda$ is a singular limit of indestructible supercompact cardinals), we do not know if such indestructibility can hold for 
$\Delta_{\aleph_{\omega^2 + 1}, \aleph_{\omega^2}}$. 
Second, the forcing that adds a thread to a $\square(\lambda^+)$ sequence always destroys stationary sets. This is problematic since it may happen that before we force the thread, the Delta reflection fails, and after the threading forcing, all the ``bad'' stationary sets that do not reflect are destroyed. In other words, the threading forcing may resurrect the Delta reflection. 

The approach that we suggest in order to get around the second problem is to force a square sequence and then define a preparatory forcing iteration. Such an iteration destroys the bad stationary sets in a way that will make the threading forcing generically preserve stationary sets. Namely, we can find for every stationary set a generic filter for the threading forcing such that in the generic extension the set remains stationary. In this way, if a given stationary set was a counterexample for the Delta reflection principle, then by using the distributivity properties of the threading forcing, we will get that there is a generic filter for the threading forcing such that in the generic extension the stationary set remains a counterexample for the Delta reflection. Therefore, it will be enough to show that the Delta principle holds after the threading forcing. 

%We do not know whether it is possible to make the Delta reflection principle indestructible under closed forcings (see Question~\ref{que: indestructible delta}). 

We are now ready to introduce some of the forcing notions that will determine our final model. 

\begin{definition} For every $n<\omega,$ we let $\mathbb{C}_n$ be the product forcing $$\prod_{m\geq n} \Coll(\kappa_m^{++}, <\kappa_{m+1})$$ (with full support). 
\end{definition}

\begin{definition} We define on $\mathbb{C}_0$ an equivalence relation $\sim:$  
$$c\sim c'\iff \exists n \forall m\geq n\ (c(m)=c'(m)).$$
Given a condition $c\in \mathbb{C}_0,$ we denote by $[c]_\sim$ the equivalence class of $c.$ We let $\mathbb{C}_{fin}$ be the forcing whose conditions are the equivalence classes of $\mathbb{C}_0$ ordered by eventual domination, namely 
$$[c]_{\sim}\leq [d]_\sim\iff \exists n\forall m\geq n\ c(m)\leq d(m)$$
\end{definition}

\begin{definition}\label{def: definition of S} $\mathbb{S}$ is the forcing that adds a $\square(\lambda^+)$-sequence with bounded approximations, namely a condition in $\mathbb{S}$ is a sequence of the form 
$\langle C_\alpha;\ \alpha\in \gamma+1\rangle$ where $\gamma<\lambda^+$ and for every $\alpha,$ $C_\alpha\subseteq \alpha$ is a club (if $\alpha$ is a successor ordinal, then $C_\alpha= \{\alpha-1\}$); for every $\alpha,\beta,$ if $\beta\in acc(C_\alpha),$ then $C_\alpha\cap \beta= C_\beta.$ 
Given two conditions $s, t,$ we say that $s$ is stronger than $t$ if $t\sqsubseteq s.$    
\end{definition}

\begin{definition}\label{def: definition of T} Let $\mathcal{C}\in V^{\mathbb{S}}$ be the generic square sequence added by $\mathbb{S}.$ We let $\mathbb{T}\in V^{\mathbb{S}}$ be the forcing that adds a thread to $\mathcal{C},$ namely a conditions in $\mathbb{T}$ is a set $C_\alpha$ in $\mathcal{C}$ where $\alpha$ is a limit ordinal and the order on $\mathbb{T}$ is defined by $p\leq q$ if and only if $q\sqsubseteq p.$
%are the elements of $\mathcal{C}$ ordered by reverse end-extension.  
\end{definition}

Note that $\mathbb{T}$ is an $\mathbb{S}$-name, but we will simply write $\mathbb{T}$ rather than $\dot{\mathbb{T}}$ to simplify the notation; we will do the same for other forcing notions, unless there is an ambiguity. 

Later we will define the following forcing notions.  

\begin{itemize}
\item $\mathbb{R}\in V^{\mathbb{C}_{fin}\ast \mathbb{S}}$ will be an iteration for destroying the stationary sets that do not reflect in $V^{\mathbb{C}_{fin}\ast \mathbb{S}}.$ $\mathbb{R}$ is defined in \autoref{subsection: generically preservation of stationary sets}.
\item $\mathbb{P}\in V$ will be a version of Magidor and Shelah construction to force the Delta reflection at $\aleph_{\omega^2+1}.$ This is a Prikry type forcing and it will be defined in \autoref{subsec: prikry forcing}. 
\item $\mathbb{P}^*$ will be a $\lambda$-directed sub-forcing of $\mathbb{P}$ depending on a generic filter for $\mathbb{C}_{fin}.$ It will be defined in \autoref{subsec: prikry forcing}. 
\end{itemize}

Our final model is obtained by forcing with  $\mathbb{C}_{fin}\ast\mathbb{S}\ast\mathbb{R}\ast\mathbb{P}^*.$ 
In order to prove that the Delta reflection holds at $\aleph_{\omega^2+1}=\lambda^+$ after this forcing, we will first show that $\lambda^+$ has the Delta reflection in the larger model  $V^{\mathbb{C}_{fin}\ast\mathbb{S}\ast\mathbb{R}\ast\mathbb{P}^*\ast\mathbb{T}},,$ then we will show that we can pull this back to our model. 

\subsection{Generic preservation of stationary sets}\label{subsection: generically preservation of stationary sets} 

In this subsection, we work in the model $W:= V^{\mathbb{C}_{fin}}$ and we let $\kappa$ be $\lambda^+.$ The results presented in this subsection do not depend on this particular choice of $W$ and $\kappa,$ they can consider to hold for any model $W$ and for any regular cardinal $\kappa$ such that $2^{<\kappa}=\kappa$ and $2^{\kappa}=\kappa^+.$ 

\begin{definition}
We say that a forcing $\mathbb{Q}$ generically preserves stationary subsets of $\kappa$ if for every stationary set $S\subseteq\kappa$ there is $q\in\mathbb{Q}$ such that $q\Vdash_{\mathbb{Q}} \check S$ is stationary.
\end{definition}

Now, we want to define the forcing notion $\mathbb{R}$ that will force $\mathbb{T}$ to generically preserve the stationary subsets of $\kappa.$ We start by stating some well known properties of $\mathbb{S}$ and $\mathbb{T}$.
 
\begin{lemma} Let $\mathbb{S}$ and $\mathbb{T}$ be as in Definitions~\ref{def: definition of S} and \ref{def: definition of T}. Then:
\begin{itemize}
\item $\mathbb{S}$ is $\sigma$-closed
\item $\Vdash_{\mathbb{S}} \square(\kappa)$.
\item $\mathbb{S}\ast\mathbb{T}$ has a $\kappa$-closed dense sub-forcing. 
\end{itemize}
\end{lemma}

\begin{proof}
First we prove that $\mathbb{S}$ is $\sigma$-closed. Let $\langle p_n \mid n < \omega\rangle$ be a strictly decreasing $\omega$-sequence of conditions. Let us denote $p_n = \tup{C_\alpha^n \mid \alpha \leq \gamma_n}$. Since $p_0 > p_1 > \cdots p_n > \cdots$ we have $\gamma_0 < \gamma_1 < \cdots$ and, for every $\alpha < \gamma_n$ and $n < m$, we have $C_\alpha^m = C_\alpha^n;$ so we can drop the superscript and simply denote this set $C_\alpha.$   
Let $\gamma_\omega = \sup_{n < \omega} \gamma_n$, and let $C_{\gamma_\omega}$ be any cofinal $\omega$-sequence at $\gamma_\omega$. 
Let $q = \tup{ C_\alpha \mid \alpha < \gamma_\omega} ^\smallfrown \tup{C_{\gamma_\omega}}$. $q$ is a condition, and it is stronger than all $p_n$. 

Let us show that there is no thread for the generic square sequence added by $\mathbb{S}.$ Let $\dot C$ be a name for a club in $W^{\mathbb{S}}$. 
We want to prove that there is an ordinal $\alpha$ and a condition $p$ such that $p \Vdash C_\alpha \neq \dot C\cap\check\alpha$ while $p\Vdash \check\alpha \in \acc \dot C$. 
Choose by induction an increasing sequence of ordinals $\beta_n$ and a decreasing sequence of conditions $p_n$ such that $p_n \Vdash \check \beta_n \in \dot C$ and $\beta_n > \dom p_{n-1}$ (for $n > 0$). 
Let $\beta_\omega = \sup \beta_n$. Every lower bound $q$ of the sequence $p_n$, forces that $\beta_\omega \in \dot C$ so if we pick $C_{\beta_\omega}$ to be a cofinal sequence which is disjoint from $\{\beta_n \mid n < \omega\}$, $q \Vdash C_{\beta_\omega} \neq \dot C \cap \check\beta_\omega$ while $q \Vdash \check\beta_\omega \in \acc \dot C$. 

Finally, we show that the following subposet of $\mathbb{S}\ast \mathbb{T}$ is dense and $\kappa$-closed:  
$$\mathcal D = \{\tup{s, \check{t}} \in \mathbb{S} \ast \mathbb{T} \mid \dom s = \max t + 1 \}.$$
First we show that $\mathcal{D}$ is dense. For every $\tup{s, t}\in \mathbb{S}\ast\mathbb{T}$ let us pick by induction a decreasing sequence of elements $\tup{s,t}\geq \cdots \geq \tup{s_{n-1}, t_{n-1}} \geq \tup{s_n, t_n}\geq \cdots $ such that $\dom s_n \geq \max t_n \geq \dom s_{n-1}$. 
Let $s_\omega:= \bigcup_{n<\omega} s_n$ and $t_\omega:= \bigcup_{n<\omega} t_n,$ then the condition $\tup{s_\omega\smallfrown t_\omega, t_\omega}$ belongs to $\mathcal{D}.$ 
To show that $\mathcal D$ is $\kappa$-closed it is enough to observe that given some $\mu < \kappa$ and a decreasing sequence of conditions $\tup{s_\alpha, t_\alpha}$, for $\alpha < \mu$, the pair $\tup{ \bigcup_{\alpha} s_\alpha ^\smallfrown \tup{\bigcup_\alpha t_\alpha}, \bigcup_\alpha t_\alpha}$ is a condition in $\mathbb{S} \ast \mathbb{T}.$ 
\end{proof} 

It follows from this lemma that in particular, $\mathbb{S}$ and $\mathbb{T}$ are $\kappa$-distributive ($\mathbb{S}$ is even $\kappa$-strategically closed). 

We are now ready to define in $W^{\mathbb{S}}$ the preparatory forcing iteration $\mathbb{R}$ that will kill the ``bad'' stationary sets, namely those that constitute a counterexample to the Delta-reflection in $W^{\mathbb{S}}$ but are destroyed by $\mathbb{T}.$ 
Intuitively, we would like to add a club disjoint from every such set, however killing one bad set might introduce new bad sets, so we have to iterate. In the iteration, we only kill stationary sets that are going to be destroyed by the threading forcing $\mathbb{T}.$ Since $\mathbb{T}$ does not collapse cardinals, this approach enables us to find a $\kappa$-closed dense subset in the iteration in some extension of the universe with the same cardinals, thus we can make sure that cardinals are not collapsed in this process. This technique for killing \emph{fragile} sets appears, for example, in \cite[Section 10]{Magidor-Cummings-Foreman-Squares}. %We use the fact that after forcing with $\mathbb{T}$ the reflection properties are restored in order to argue that all bad sets were killed in the process.

For technical reasons, we would like this preparatory forcing iteration to be homogeneous. This means that we cannot simply iterate over bad stationary sets, since their stationarity might depend on the generic filter for the previous steps. Thus, in the iteration we kill the subsets of $\kappa$ which are forced to be non-stationary by the maximal condition of $\mathbb{T}$. In particular, we will add a Cohen subset of $\kappa$ at cofinally many steps.

Recall that for a given subset $A$ of $\mu$ (where $\mu$ is a cardinal), there is a forcing notion that adds a club $\dot{C}$ disjoint from $A:$ the conditions of this forcing are closed sequences of length less than $\mu$ of ordinals in 
$\mu\setminus A;$ the order is the end-extension. We refer to this forcing notion as the forcing that ``shoots a club disjoint from $A$''.

\begin{definition} In $W^{\mathbb{S}},$ we define $\mathbb{R}$ as an iteration of length $\kappa^+$ with support of size less than $\kappa$ of a certain sequence $\langle \dot{\mathbb{Q}}_\alpha;\ \alpha<\kappa^+\rangle$ (yet to be defined).  At each stage we will have $\force_{\mathbb{R}_\alpha}\ \vert \dot{\mathbb{Q}}_\alpha\vert =\kappa$ (where $\mathbb{R}_\alpha$ is the iteration up to stage $\alpha$) thus $\mathbb{R}$ will be $\kappa^+$-c.c.   
Every $\mathbb{R}_\alpha$-name $\dot{A}$ for a subset of $\kappa$ can be represented by $\kappa$ many antichains of $\mathbb{R}_\alpha$ and, as every such antichain has size less than $\kappa^+,$ there are at most $(\kappa^+)^{\kappa}$ many such antichains, hence since $2^{\kappa}=\kappa^+,$ there are at most $\kappa^+$ many $\mathbb{R}_\alpha$-names for subsets of $\kappa.$ Let $f$ be a function from $\kappa^+$ onto $\kappa^+\times \kappa\times \kappa$ such that if $f(\alpha)=(\beta, \gamma, \delta),$ then $\beta<\alpha.$ 
We now define the $\dot{\mathbb{Q}}_\alpha$'s inductively as follows. 
Suppose $\langle \dot{\mathbb{Q}}_\beta;\ \beta<\alpha\rangle$ has been defined, let $f(\alpha)=(\beta,\gamma, \delta).$ Let $\dot{A}$ be the $\gamma$-th $\mathbb{R}_\beta$-name for a subset of $\kappa.$ If 
$$\langle 1_{\mathbb{S}\ast \mathbb{R}_\alpha}, 1_\mathbb{T} \rangle\force^W\ \dot{A} \textrm{ is non-stationary },$$
(namely the maximal condition of $\mathbb{T}$ forces over $W^{\mathbb{S}\ast \mathbb{R}_\alpha}$ that $\dot{A}$ is non-stationary),  
then we let $\dot{A}_{\alpha}=\dot{A},$ otherwise we let $\dot{A}_{\alpha}$ be the empty set. We let $\dot{\mathbb{Q}}_\alpha$ be (an $\mathbb{R}_\alpha$-name for) the forcing that shoots a club disjoint from $\dot{A}_\alpha.$ 
%otherwise, we let $\dot{Q}_\alpha$ be the Cohen forcing for adding a new subset of $\kappa.$     
\end{definition}
 
\begin{lemma} $W^{\mathbb{S}\ast\mathbb{R}}\models\ \mathbb{T}$ generically preserves stationary subsets of $\kappa$
\end{lemma}

\begin{proof} Let $A\subseteq \kappa$ be a stationary set in $W^{\mathbb{S}\ast\mathbb{R}}.$ Suppose by contradiction that the maximal condition of $\mathbb{T}$ forces that $A$ is not stationary. In $W^{\mathbb{S}\ast\mathbb{R}}$, let $\dot{C}$ be a $\mathbb{T}$-name for a club in $\kappa$ disjoint from $A.$  Since the iteration is $\kappa^+$-c.c.,  $\dot{C}$ belongs to $W^{\mathbb{S}\ast\mathbb{R}_\beta}$ for some $\beta < \kappa^+.$ We chose $f$ in such a way that every $\mathbb{S}\ast \mathbb{R}$-name for a subset of $\kappa$ will be considered unboundedly often during the iteration. Thus, we have an ordinal $\gamma > 
\beta$ in which $\dot{A}_\gamma$ is a name for $A$. Therefore $\mathbb{Q}_\gamma$ is the forcing that adds a club disjoint from $A$, and $A$ is not stationary in $W^{\mathbb{S}\ast\mathbb{R}},$ a contradiction. 
\end{proof}

The main issue is why the iteration $\mathbb{R}$ does not collapse cardinals:

\begin{lemma}\label{lemma: SRalphaT contains a kappa closed dense subset}
For all $\alpha \leq \kappa^+$, $\mathbb{S}\ast\mathbb{R}_\alpha \ast \mathbb{T}$ contains a $\kappa$-closed dense subset ($\mathbb{R}_{\kappa^+}$ is $\mathbb{R}$). 
\end{lemma}

\begin{proof}
First, note that since $\mathbb{T}\in W^\mathbb{S},$ we have $\mathbb{S}\ast\mathbb{R}_\alpha\ast\mathbb{T} \cong \mathbb{S}\ast\mathbb{T} \ast \mathbb{R}_\alpha$.

The proof is by induction on $\alpha$. We will prove by induction a slightly stronger assertion: for every $\alpha \leq \kappa^+$ there is a dense subset $D_\alpha\subseteq\mathbb{S}\ast\mathbb{R}_\alpha \ast \mathbb{T}$ such that:
\begin{itemize}
\item $D_\alpha$ is $\kappa$-closed, and in fact every decreasing sequence of length $<\kappa$ in $D_\alpha$ has a unique maximal lower bound. 
\item if $\alpha < \beta$ , $p \in D_\beta$ and $q = p \restriction \alpha$ then $q \in D_\alpha$. 
\item if $\alpha < \beta$ , $p \in D_\alpha$, $q \restriction \alpha = p$ and for every $\alpha \leq \gamma < \beta$, $q(\gamma) = 1$ then $q \in D_\beta$.
\item if $\beta$ is a limit ordinal, $q \in \mathbb{S}\ast\mathbb{R}_\beta\ast\mathbb{T}$ and for every $\alpha < \beta$ $q\restriction \alpha \in D_\alpha$ then $q\in D_\beta$.  
\end{itemize}
For $\alpha = 0$ - this is the previous lemma. 
Assume now that the lemma is true for every $\delta < \alpha$. 

If $\alpha$ is a limit ordinal then the set $D_\alpha$ is determined by our last assumption. We need to show that it is dense and $\kappa$-closed. Let $q_0 \in \mathbb{S}\ast\mathbb{R}_\alpha\ast\mathbb{T}$ be any condition and let us find an element in $D_\alpha$ below it. We separate the proof into two cases: first assume that $\cf \alpha < \kappa$. In this case, let $\{\gamma_i \mid i < \cf \alpha\}$ be a continuous cofinal sequence with $\gamma_0 = 0$. We pick by induction on $i < \cf \alpha$ an element $q_i \in \mathbb{S}\ast\mathbb{T}\ast\mathbb{R}_{\gamma_i}$ such that $q_{i + 1}\restriction \gamma_i \leq q_i$, $q_{i + 1} \in D_{\gamma_{i + 1}}$. Note that by our assumption on the sets $D_\beta$, for every limit ordinal $i \leq \cf \alpha$, and $\beta < \gamma_i$, the sequence $q_j \restriction \beta$ for $j < i$ such that $\gamma_j \geq \beta$  is a decreasing sequence in $D_\beta$ so it has a limit. Let $q_i$ be defined in this case as the condition in which for every $\beta < \gamma_i$, $q_i \restriction \beta$ is this unique maximal lower bound. By uniqueness, those lower bounds cohere. 
For the second case, suppose $\cf \alpha \geq \kappa$ then for every $q \in \mathbb{S}_\alpha$, $\dom q$ is bounded below $\alpha$, by some ordinal $\delta$. Let $r \leq q$ in $D_\delta$ then the extension of $r$ with the maximal conditions of the stages between $\delta$ and $\alpha$ is a condition in $D_\alpha$ below $q$. 

Let $\alpha = \beta + 1$ and suppose that $D_\beta$ is defined. We want to define $D_\alpha$. By our choice of $\dot{A}_\beta$, after forcing with $\mathbb{T}$ and $\mathbb{S}\ast\mathbb{R}_\beta$, there is a club disjoint from $\dot{A}_\beta$. Let us fix an $\mathbb{S}\ast \mathbb{R}_{\beta}\ast \mathbb{T}$-name for this club, $\dot{E}$. We define $D_\alpha$ to be the set of all conditions $q$ such that $q\restriction \beta \in D_\beta$, $q(\beta)\in W$ and $q\restriction \beta \Vdash \max q(\beta) \in \dot{E}$. 

We show that $D_\alpha$ is dense. Let $q \in \mathbb{S}\ast\mathbb{T}\ast\mathbb{R}_\alpha$. Using the $\kappa$-closure of $D_\beta$ we may assume that $q(\beta) \in W$. Let us extend $q\restriction \beta$ and decide the next element of $\dot{E}$ above $\gamma:= \max q(\beta).$ By further extension we may assume again that $q\restriction \beta \in D_\beta$. Let us extend $q(\beta)$ by adding $\gamma$ in its top. Clearly now $q \in D_\alpha$.

Let us show that $D_\alpha$ has the desired closure property: let $\eta < \kappa$ and let $\tup{ q_i \mid i < \eta}$ be a descending sequence of elements in $D_\alpha$. In particular, $q_i \restriction \beta \in D_\beta$, so it has a unique maximal lower bound, $r$. Moreover, for every $i,$ we have $r \Vdash \max q_i(\beta) \in \dot{E}$ and therefore if $\gamma = \sup \bigcup q_i (\beta),$ then $r\Vdash \check \gamma \in \dot{E}$. This means that $r\Vdash \check \gamma \notin \dot{A}_\beta$ so $s = \bigcup_{i<\eta} q_i (\beta) \cup \{\gamma\}$ is forced by $r$ to be a condition in $\mathbb{Q}_\beta$, and it is clear that $s^\smallfrown \tup{r}$ is the maximal lower bound of $\tup{ q_i \mid i < \eta}$.
\end{proof}

\begin{remark} 
Note that the dense sets $D_\alpha$ are in fact $\kappa$-\emph{directed} closed, thus $\mathbb{S}\ast \mathbb{R}_\alpha\ast \mathbb{T}$ contains a $\kappa$-directed closed dense subset. 
\end{remark}

\begin{coroll}\label{coroll: directed closure Cfin} The following hold. 
\begin{enumerate}
\item $\mathbb{S}\ast\mathbb{R}\ast\mathbb{T}$ is equivalent to a $\kappa$-directed closed forcing;
\item $\mathbb{R}$ is $\kappa$-distributive;
\item $\mathbb{W}^{\mathbb{S}\ast \mathbb{T}}\models \mathbb{R}\textrm{ contains a $\kappa$-closed dense subset}$
\end{enumerate}
\end{coroll}

Now we show that some forcing notions preserve $\square(\kappa).$

\begin{lemma}\label{lemma: preserving the coherent sequence}
Let $\mathcal{C}$ be a coherent sequence and let $\mathbb{B}$ be a forcing notion such that $\mathbb{B}\times\mathbb{B}\Vdash \cf \kappa \geq \omega_1$. Then $\mathbb{B}$ does not add a new thread to $\mathcal{C}$.
\end{lemma}

\begin{proof}
Let $\mathcal{C} = \langle C_\alpha \mid \alpha < \kappa\rangle$. Suppose for a contradiction that there is a new thread $\dot{D}.$ Let $G\times H$ be a generic filter for $\mathbb{B}\times \mathbb{B}$. Since $\Vdash \dot{D}\notin W$, we have $\dot{D}^G\neq \dot{D}^H$. 
Let $\beta:= \min (\dot{D}^G\Delta \dot{D}^H)$ (i.e. the least ordinal in the symmetric difference). Since $\dot{D}^G, \dot{D}^H$ are both clubs at $\kappa$ and $\cf \kappa > \omega$ in $W[G\times H]$, there is $\alpha \in \acc \dot{D}^G \cap \acc \dot{G}^H$ above $\beta$. But, by the definition of thread, $\dot{D}^G\cap \alpha = C_\alpha = \dot{D}^H\cap \alpha,$ thus in particular $\beta\in \dot{D}^G\cap \dot{D}^H,$ contradicting our choice of $\beta.$ \end{proof}

\begin{lemma}\label{lemma: preserving the square sequence}
Let $\mathcal{C}$ be a $\square(\kappa)$ sequence and let $\mathbb{B}$ be a $\kappa$-c.c.\ forcing notion. Then $\mathcal{C}$ is a $\square(\kappa)$ sequence in $W^{\mathbb{B}}$.
\end{lemma}

\begin{proof}
Clearly, $\mathcal{C}$ remains coherent in $W^{\mathbb{B}}$ so the only think that we need to verify is that $\mathcal{C}$ still does not have a thread. 
Assume otherwise, and let $\dot{D}$ be a name for the thread and denote $\mathcal{C} = \langle C_\alpha \mid \alpha < \kappa\rangle$.
By the chain condition, there is a club $E\in W$ such that $\Vdash_{\mathbb{B}} \check{E}\subseteq\dot{D}$. But for every $\alpha < \beta$ in $\acc E$, $\Vdash \alpha, \beta \in \acc \dot{D}$ and therefore $\Vdash \dot{D}\cap \check\alpha = \check{C_\alpha}$, $\Vdash \dot{D}\cap \check\beta = \check{C_\beta}$ and in particular, it is forced that $C_\alpha \trianglelefteq C_\beta$, so $T = \bigcup_{\alpha \in \acc E} C_\alpha$ is a thread.    
\end{proof}

\begin{coroll}\label{coroll: existence of square sequence}
Let $\mathcal{C}$ be the generic square sequence added by $\mathbb{S}$. Then $\mathcal{C}$ remains a square sequence even after forcing with $\mathbb{R}.$  
\end{coroll}

\begin{proof}
By Corollary~\ref{coroll: directed closure Cfin}, in $W^{\mathbb{S}\ast\mathbb{T}}$, $\mathbb{R}$ contains a $\kappa$-closed dense subset (just take the set of conditions from the dense set $D_{\kappa^+}$ such that their first coordinates are in the generic filter for $\mathbb{S}\ast\mathbb{T}$). In particular, $\mathbb{R}\times \mathbb{R}$ is equivalent to a $\kappa$-closed forcing, hence it does not change the cofinality of $\kappa$ in $W^{\mathbb{S\ast T}}$ and, in particular, it does not change the cofinality of $\kappa$ in $W^{\mathbb{S}}$. 

We conclude that in $W^{\mathbb{S}}$, $\mathbb{R}\times\mathbb{R}\Vdash \cf \kappa = \kappa > \omega$, and therefore by Lemma~\ref{lemma: preserving the coherent sequence}, it cannot add a new thread for $\mathcal{C}$. 
\end{proof}

We showed that $\mathbb{T}$ generically preserves stationary subsets of $\kappa$ in $W^{\mathbb{S}\ast\mathbb{R}}.$ 
It is useful to show that generically preservation of stationary sets of $\kappa$ is preserved under $\kappa.$-c.c.\ forcing notions. We will use such result in the proof of the main theorem
%, and the proof will be a "toy example" for a technique that we will use %during the main theorem's proof.  

\begin{lemma}\label{lem: c.c. preserves generically stationary preservation}
Let $V_0$ be a model of \ZFC, let $\mathbb{T}\in V_0$ be a forcing notion that generically preserves any stationary subset of $\kappa,$ and let $\mathbb{B}\in V_0$ be a forcing notion such that $\Vdash_{\mathbb{T}} \check{\mathbb{B}}$ is $\check\kappa$-c.c. 
Then $$V_0^{\mathbb{B}}\models\ \mathbb{T}\textrm{ generically preserves stationary subsets of $\kappa$}.$$   
\end{lemma}
\begin{proof}

Work in $V_0,$ let $\dot E$ be a $\mathbb{B}$-name for a stationary subset of $\kappa$, and let $H$ be a generic filter for $\mathbb{B}$. Assume by contradiction that the maximal condition of $\mathbb{T}$ forces (over $W$) that $\dot E^H$ (the realization of $\dot E$ according to the generic filter $H$) is non-stationary, and let $\dot C$ be a $\mathbb{B}\times \mathbb{T}$-name for a club disjoint from $\dot E$. Since $\mathbb{B}\times \mathbb{T} \cong \mathbb{T}\times \mathbb{B}$ and $\mathbb{B}$ is forced to be $\kappa$-c.c.\ after $\mathbb{T}$, we can find also a $\mathbb{T}$-name for a club $\dot D$ that is forced (by the maximal condition of $\mathbb{B}\times \mathbb{T}$) to be a sub-club of $\dot C$. 

We define a fake version of $\dot E$ by letting $E^\star = \{\alpha < \kappa \mid \exists q\in \mathbb{B},\,q\Vdash \alpha \in \dot E\}$. Clearly, $E^\star$ is in the ground model and it is forced to cover $\dot E$. Moreover, $\dot D$ is forced (by the maximal condition) to be disjoint from $E^\star$, as otherwise, there is some $t\in T$ and $\alpha \in E^\star$ such that $t\Vdash \check \alpha \in \dot D$ and by the definition of $E^\star$ there is also some $q\in \mathbb{B}$ that forces $\alpha \in \dot E$. But then $\langle q, t\rangle \Vdash_{\mathbb{B}\times \mathbb{T}} \check \alpha \in \dot E \cap \dot D \subseteq \dot E \cap \dot C = \emptyset.$

Now we may apply the property of $\mathbb{T}$ in the ground model and conclude that $E^\star$ is non-stationary in the ground model, and since it covers $\dot E$ we conclude that $\dot E$ is forced to be non-stationary. \end{proof}

The following technical lemma will enables us to control the distributivity of $\mathbb{T}$ after some forcing extensions. Recall that, given a forcing notion $\mathbb{Q}$ and a $\mathbb{Q}$-name $\dot{A}$ for a subset of an ordinal $\alpha,$ we say that $\dot{A}$ is a \emph{fresh subset} of $\alpha$ if $1_\mathbb{Q}\force \dot{A}\notin V$ and, for every $\beta<\alpha,$ 
$1_\mathbb{Q}\force \dot{A}\cap \beta\in V.$

\begin{lemma} Let $\mathbb{T}$ be as above and let $\mathbb{B}$ be a forcing notion such that $\mathbb{T}\times \mathbb{B}$ does not change the cofinality of $\kappa$ then $\Vdash_{\mathbb{B}} \check{\mathbb{T}}$ is $\kappa$-distributive. 
\end{lemma}
\begin{proof}
Work in $W^\mathbb{B}.$ Suppose by contradiction that there is a $\mathbb{T}$-name $\dot \tau$ for a new subset of ordinals of cardinality $\eta < \kappa$. Let $G$ be a generic filter for $\mathbb{T}$. We will think of $G$ as a fresh subset of $\kappa$ (namely, the generic thread). Since $G$ is generic, it realizes $\dot \tau$. For every $\alpha < \eta$, let $\beta_\alpha < \kappa$ be an ordinal such that the condition $G\cap \beta_\alpha$ decides the value of $\dot \tau (\alpha)$. By the regularity of $\kappa$ in the generic extension there is a bound $\beta < \kappa$, $\beta \geq \beta_\alpha$ for every $\alpha < \eta$. But then $G\cap \beta$ (which is a condition in $\mathbb{T}$) already decided all the values of $\dot \tau$, so it forced that $\tau \in W$.    
\end{proof}

%We will use this lemma with $\mathbb{Q}= \mathbb{P}^*$ (that will be defined in the next subsection) which is a $\kappa$-c.c. forcing.  
% In our case we will force with a forcing that is $\kappa$-c.c.\ in $W^{\mathbb{S}\ast\mathbb{R}\ast\mathbb{T}}$, so the conditions of the lemma hold.   

%Before we continue to define the main forcing, let us note that 
We can already prove that forcing with $\mathbb{T}$ will have no effect on the Delta reflection at $\lambda^+$. 

\begin{lemma}\label{lem: avoiding T}
Let $\mathbb{T}$ be as above - a $\kappa$-distributive forcing that generically preserves stationary sets at $\kappa$ and let $\mu < \kappa$. If $\mathbb{T}$ forces with the maximal condition that $\Delta_{\mu, \kappa}$ holds, then it holds in the ground model as well.
\end{lemma}

\begin{proof}
Let us pick in the ground model a stationary subset of $\kappa$, $S$ and an algebra $A$ with less than $\mu$ operations. We want to find a sub-algebra $B$ of $A$ such that $B$ as order type a regular cardinal below $\mu$ and $S\cap B$ is stationary at $\sup B.$ We know that $\Delta_{\mu,\kappa}$ holds after forcing with $\mathbb{T}.$ Let $t\in\mathbb{T}$ be a condition that forces that $S$ is still stationary. So a generic filter that contains $t$ must introduce a sub-algebra $B$ with regular order type such that $B \cap S$ is stationary at $\sup B$. Since $\mathbb{T}$ is $\kappa$-distributive $B$ appears also in the ground model. The fact that $B$ is a sub-algebra is absolute. The regularity of its order type as well as the stationarity of $S\cap B$ is downward absolute, so since it holds in the generic extension, it must hold in the ground model as well.     
\end{proof}

We will use Lemma \ref{lem: avoiding T} with two forcing notions: $\mathbb{P}^*$ and ${\mathbb{C}_n}/{\mathbb{C}_{fin}},$ where $\mathbb{P}^*$ will be defined in the subsection~\ref{subsec: prikry forcing}. Both forcings are defined in $V^{\mathbb{C}_{fin}}$ and we will prove that they are $\lambda^+$-c.c.\ in this model. This will be enough as the iteration $\mathbb{S}\ast\mathbb{R}\ast\mathbb{T}$ contains a $\lambda^+$-closed dense set and therefore it cannot introduce an antichain of size $\lambda^+$ to a $\lambda^+$-c.c.\ forcing notion. This means that those forcing notions will have the desired chain condition after forcing with $\mathbb{T}$ so $\mathbb{T}$ is distributive and generically stationary preserving also after them. 

In the process of defining the forcing notion that will introduce the $\Delta$-reflection principle at $\aleph_{\omega^2 + 1}$ we will need to pick normal ultrafilters from $V$ that are projections of ultrafilters in the generic extension. We need this choice to be independent on the generic filters. 

\begin{lemma}\label{lemma: SRT weakly homogeneous} $\mathbb{S}\ast \mathbb{R}\ast \mathbb{T}$ is weakly homogeneous.
\end{lemma}

\begin{proof}
Let $p, p^\prime \in \mathbb{S}\ast \mathbb{R}\ast \mathbb{T}$. We want to find $q\leq p,\,q^\prime \leq p^\prime$, dense sets $D, D^\prime$ below $q, q^\prime$ respectively and an automorphism $\pi\colon D\to D^\prime.$ First, let us ignore $\mathbb{R}$ and work with $\mathbb{S}\ast \mathbb{T}$. Let $\langle s,t\rangle\in \mathbb{S}\ast\mathbb{T}$, $\langle s^\prime,t^\prime \rangle\in \mathbb{S}\ast\mathbb{T}$ and assume, without loss of generality, that $\dom s = \dom s^\prime = \gamma + 1$, and $\max t = \max t^\prime = \gamma$ and $t, t^\prime \in W$. Let us define an automorphism, $\pi$, between the conditions below $\langle s, t\rangle$ and the conditions below $\langle s^\prime, t^\prime \rangle$. Let $\langle a, b\rangle \leq \langle s, t\rangle$ and assume that $\max \dom a = \max b + 1 = \delta + 1$, $b\in V$. We want to define $\langle a^\prime, b^\prime\rangle = \pi(a,b)$. For every $\rho \leq \delta$, let $\xi = \max (\acc a_\rho \cap (\gamma + 1))$, if $\acc a_\rho \cap (\gamma + 1) \neq \emptyset$, and set $a^\prime_\rho = a_\rho \setminus \xi \cup s^\prime_\xi$. Otherwise, set $a^\prime_\rho = a_\rho$. Let $b^\prime = b\setminus \gamma \cup t^\prime.$ Note that $a^\prime$ is a coherent sequence, $b^\prime$ is the last element of $a^\prime$, and it is clear that if $\langle a_1, b_1 \rangle \leq \langle a_0, b_0\rangle$ then $\pi(\langle a_1, b_1\rangle ) \leq \pi(\langle a_0, b_0 \rangle)$. It follows from a theorem of Vop\v{e}nka and H\'ajek \cite{VopenkaHajek} that if we have a partial automorphism defined below some condition, we can extend it to an automorphism of the whole Boolean completion of the forcing. Therefore, we can extend $\pi$ to an automorphism defined over the Boolean completion of $\mathbb{S}\ast \mathbb{T}$ (from now on we will not distinguish between a poset and its Boolean completion).  

Now we want to extend $\pi$ to an automorphism of $\mathbb{S}\ast\mathbb{R}\ast\mathbb{T}$. For that, we define in $W$ a permutation of the coordinates of the iteration which depends on $\pi$, that will witness the automorphism. For simplicity, let us start by dealing with the first step of the iteration, namely with the forcing $\mathbb{Q}_0$ that shoots a club disjoint from $\dot{A}_0$:

By the definition of $\mathbb{R}$, $\dot{A}_0$ is a $\mathbb{S}$-name for a subset of $\kappa$ that is forced by the maximal condition of $\mathbb{S}\ast \mathbb{T}$ to be non stationary. Let $\pi(\dot{A}_0)=\{( \pi(u), \check\alpha) \mid u\in \mathbb{S}\ast\mathbb{T},\,u \Vdash \check\alpha\in \dot{A}_0\}$. Since the maximal condition of $\mathbb{S}\ast \mathbb{T}$ is not moved by $\pi,$ it forces $\pi(\dot{A}_0)$ to be non stationary. In particular, $p \Vdash $ $``\pi(\dot{A}_0)$ is non-stationary"

Therefore, there is $\alpha < \kappa^+$ such that $\pi(\dot{A}_0) = \dot{A}_\alpha$. Note that by assuming that the names $\dot{A}_\alpha$ are canonical (namely, that $\dot{A}_\alpha \subseteq (\mathbb{S}\ast\mathbb{R})\times\kappa$), we get that $\alpha$ is determined by $p$ itself, and we don't need to extend it. 

We are now ready to define by induction a permutation of $\kappa^+$, $\pi^*$, and an extension of the automorphism $\pi$ that was defined previously. In the inductive process we will construct subsets 
$\{I_\rho\}_{\rho<\kappa^+}$
of $\kappa^+$ which are not necessarily initial segments, and we will define the automorphism on those sets, so we will also need to verify that the restrictions of the iteration to those sets are well defined forcing notions. Formally, we will treat only conditions $r\in\mathbb{R}$ such that $r_i \in W$ (namely, $r_i$ is the canonical name for a set in $W$) for every $i\in I_\rho$, and $r\restriction I_\rho$ is the condition obtained by replacing all the coordinates of $r$ which are not in $I$ with $1_{\mathbb{Q}_\gamma} = \emptyset$.  

We define inductively sets $I_\rho, J_\rho\subseteq \kappa^+$ and a map $\pi_\rho: \mathbb{S}\ast\mathbb{R}\restriction I_\rho \ast \mathbb{T} \to \mathbb{S}\ast\mathbb{R}\restriction J_\rho \ast \mathbb{T}.$

Let $I_0 = J_0 = \emptyset$, $\pi_0 = \pi$. 

Let us assume by induction that we have defined a bijection $\pi^*\colon I_\rho\to J_\rho$ such that:
\begin{enumerate}
\item $|I_\rho|, |J_\rho| < \kappa^+$.
\item $\forall \alpha\in I_\rho,\, s\Vdash \dot{A}_\alpha$ is a $\mathbb{S}\ast\mathbb{R}\restriction I_0\cap \alpha$-name. 
\item The same holds for $J_\rho$.
\item The map $\pi_\rho \colon \mathbb{S}\ast\mathbb{R}\restriction I_\rho \ast \mathbb{T} \to \mathbb{S}\ast\mathbb{R}\restriction J_\rho \ast \mathbb{T}$ defined by 
$$\pi_\rho (a, b, \langle r_i \mid i \in I_\rho\rangle) = \pi_0 (a, b)^\smallfrown \langle r_{\pi^* (i)} \mid i \in I_\rho\rangle$$
 is an automorphism.  
\end{enumerate}
As we remarked above, we identify the forcing $\mathbb{S}\ast\mathbb{R}\restriction I_\rho\ast \mathbb{T}$ with its dense set of elements in which $r_i \in W$ for every $i$. 

Let $\alpha = \min \kappa^+\setminus I_\rho$. We want to extend $I_\rho$ to $I_{\rho + 1}$ by adding $\alpha$ and dealing with the consequences. By our analysis above, as $\dot{A}_\alpha$ is forced to be non stationary by the empty condition of $\mathbb{T}$ also $\pi_\rho(\dot{A}_\alpha)$ (which is well defined) is forced to be the same, so there is $\beta < \kappa^+$ such that $\pi_\rho(\dot{A}_\alpha) = \dot{A}_\beta$. Note that the set $\pi_\rho(\dot{A}_\alpha)$ must be equal (as a set in $W$) to some $\dot{A}_\beta$ since it is forced by the empty condition to be a subset of $\kappa$ which is forced by the empty condition of $\mathbb{T}$ to be non-stationary.  

We can assume that $\beta\notin J_\rho$, as every canonical name of a subset of $\kappa$ that is forced to be non-stationary after $\mathbb{T}$ will appear during the iteration unboundedly often. Let us define $\pi^*(\alpha) = \beta$, $I_{\rho + 1} = I_\rho\cup\{\alpha\}$, $J_{\rho + 1} = J_\rho \cup \{\beta\}$ and verify that all our requirements still hold. First, it is clear that $I_{\rho+1}$ and $J_{\rho + 1}$ have cardinality $|\rho + 1| < \kappa^+$. Next, if: 
$$p\restriction I_\rho \Vdash \check{r_\alpha}\cap \dot{A}_\alpha = \emptyset$$
then applying $\pi_\rho$ we get: 
$$\pi_\rho(p\restriction I_\rho) \Vdash \check{r_\alpha}\cap \pi_\rho(\dot{A}_\alpha) = \emptyset$$
Using the fact that $\pi_\rho(\dot{A}_\alpha) = \dot{A}_\beta$ we get that $\pi_\rho(p\restriction I_\rho)\Vdash \check{r_\alpha} \in \mathbb{R}_\beta$, and that $\mathbb{R}_\beta$ is a $\mathbb{P}\restriction J_\rho$-name. It is clear that this function is order preserving and reversible. 

Similarly, we extend $\pi^\star$ so that its range will contain $\min \kappa^+\setminus J_\rho$.

For limit stage $\rho$, we set $I_\rho = \bigcup_{\gamma<\rho} I_\gamma$, $J_\rho = \bigcup_{\gamma<\rho} J_\gamma$ and define the bijection $\pi^*$ as the union of all previous bijections. 
\end{proof}  

What we will actually need in the proof of the main theorem is the following lemma that is proven with the same arguments. 

\begin{lemma}\label{lem: iteration is homogeneous} For every $n<\omega,$
the forcing $\mathbb{C}_n \ast \mathbb{S} \ast\mathbb{R} \ast\mathbb{T}$ is homogeneous in $V.$ 
\end{lemma}

\begin{proof}
The proof is essentially the same as for Lemma~\ref{lemma: SRT weakly homogeneous}. $\mathbb{S}\ast \mathbb{T}$ is defined in $V$ (since it is defined in $V^{\mathbb{C}_{fin}}$ and $\mathbb{C}_{fin}$ does not add bounded subsets to $\kappa$ by Lemma~\ref{lem: c_fin is distributive}). $\mathbb{C}_n$ is weakly homogeneous and as we proved in the previous lemma, $\mathbb{S}\ast\mathbb{T}$ is weakly homogeneous. The product of two weakly homogeneous forcing is weakly homogeneous so $\mathbb{C}_{n}\ast \mathbb{S}\ast\mathbb{T} = \mathbb{C}_n \times (\mathbb{S}\ast\mathbb{T})$ is weakly homogeneous. In the above proof, in order to extend the automorphism of $\mathbb{S}\ast\mathbb{T}$ to an automorphism of $\mathbb{S}\ast\mathbb{T}\ast \mathbb{R}$ we only required that $\mathbb{S}\ast\mathbb{T}$ is weakly homogeneous, so the same argument works for $\mathbb{C}_n\ast\mathbb{S}\ast\mathbb{T}$ as well. 
\end{proof}

\subsection{The Magidor and Shelah's forcing}\label{subsec: prikry forcing}

In order to get the Delta reflection at $\aleph_{\omega^2+1},$ we will force with a version of the forcing defined by Magidor and Shelah in \cite{MagidorShelah}; we will make minor adjustments in the definition of this forcing. This construction is a version of diagonal Prikry forcing which is quite complex. In order to get a better control over the sets that are introduced by it, we will first split it into two steps - the first is $\mathbb{C}_{fin}$ which is $\lambda^+$-distributive, and the second, which we denote by $\mathbb{P}^*$, will be $\lambda$-centred. Moreover, $\mathbb{C}_n$ projects onto $\mathbb{C}_{fin}$ and the quotient forcing is $\lambda^+$-c.c. (even after further forcing), which will be useful as $\mathbb{C}_n$ also introduces a part of the generic for $\mathbb{P}^*$ while being $\kappa_n^{++}$-directed closed, so it preserves the supercompactness of $\kappa_n$ and does not add any new subsets of size $\kappa_n$.   

\begin{remark}\label{remark: projection from c_fin to c_n} There is a projection, $\iota$, from $\mathbb{C}_n$ onto $\mathbb{C}_{fin}$. 
%In particular the iteration $\mathbb{R}$ is defined in $V^{\mathbb{S}\times \mathbb{C}_n}.$
\end{remark}

Indeed, we can define $\iota$ by letting $\iota(p) = [\tup{ 1 \mid i < n}^\smallfrown p]_\sim$.

\begin{lemma}\label{lem: c_fin is distributive}
$\mathbb{C}_{fin}$ is $\lambda^+$-distributive.
\end{lemma}

\begin{proof}
For every $n$, $\mathbb{C}_{fin}$ is $\kappa_n^{++}$-distributive as a projection of $\mathbb{C}_n$. Therefore, $\mathbb{C}_{fin}$ is $\lambda$-distributive. Since the distributivity of a forcing notion is always a regular cardinal, and $\cf \lambda = \omega$, $\mathbb{C}_{fin}$ is $\lambda^+$-distributive. 
\end{proof}

As we remarked at the beginning of the subsection~\ref{subsection: generically preservation of stationary sets}, $\mathbb{S}\ast\mathbb{R}$ and $\mathbb{T}$ are defined in $V^{\mathbb{C}_{fin}}$. 
The projection $\iota$ enables us to pull those definitions into $V^{\mathbb{C}_n}$. 
Recall that $\mathbb{S}$ is defined in the same way over $V$ and over $V^{\mathbb{C}_{fin}}$ (by Lemma~\ref{lem: c_fin is distributive}). 
We will use this fact and won't distinguish between $\mathbb{S}\times \mathbb{C}_{fin}$ and $\mathbb{C}_{fin}\ast \mathbb{S}$. 
The former notation will be used when dealing with the forcing $\mathbb{C}_n$ in order to stress that we are using in this case the forcing $\mathbb{S}$ as defined in $V$ and not in $V^{\mathbb{C}_n}$.
In contrast to the situation of $\mathbb{S}$, the definition of $\mathbb{R}$ depends on the generic filter for $\mathbb{C}_{fin}$. We stress that when writing $\mathbb{C}_n \ast \mathbb{S}\ast \mathbb{R}$, $\mathbb{R}$ is still defined over $\mathbb{C}_{fin}\ast\mathbb{S}$. 

\begin{lemma}\label{lemma: directed closed} $(\mathbb{S}\times \mathbb{C}_n)\ast \mathbb{R}\ast \mathbb{T}$ is equivalent to a $\kappa_n^{++}$-directed closed forcing.
\end{lemma}
\begin{proof} 
By Corollary~\ref{coroll: directed closure Cfin}, the forcing notion $\mathbb{S}\ast\mathbb{R}\ast\mathbb{T}$ is equivalent to a $\lambda^+$ directed closed forcing in $V^{\mathbb{C}_{fin}}$. Since $V^{\mathbb{C}_n}$ agrees with $V^{\mathbb{C}_{fin}}$ about sequences of ordinals of length $< \kappa_n^{++}$ and $\mathbb{S}\ast\mathbb{R}\ast\mathbb{T} \in V^{\mathbb{C}_{fin}}$, every directed subset of $\mathbb{S}\ast\mathbb{R}\ast\mathbb{T}$ of size $<\kappa_n^{++}$ in $V^{\mathbb{C}_n}$ also appears in $V^{\mathbb{C}_{fin}}$ so it has a lower bound, as required.  
\end{proof}

We denote by $W_n$ the model $V^{(\mathbb{S}\times \mathbb{C}_n)\ast \mathbb{R} \ast\mathbb{T}}$

%\begin{lemma}\label{lemma: generic preservation of stationary sets for n} If $W_n\models S\subseteq \lambda^+\textrm{ is stationary},$ then  $W_n\models ``1_\mathbb{T}\not \force\ S\textrm{ is non stationary }''.$
%\end{lemma}

%\begin{proof} work in progress... \end{proof}

Now we give the definition of the forcing construction of Magidor and Shelah with minor modifications that are needed for our purposes. We denote this forcing notion by $\mathbb{P}.$ In order to define $\mathbb{P}$ we need some preparation (the only substantial difference with the original forcing by Magidor and Shelah is in the following preparation). 

Lemma \ref{lemma: directed closed} implies that for every $n<\omega,$
$$W_n \models\ \kappa_n\textrm{ is supercompact }.$$  
Therefore, there is in this model a normal ultrafilter $U_n^*$ on $\mathcal{P}_{\kappa_n}(\lambda^+)$ and $U_n^*$ has a natural projection to a normal ultrafilter $U_n$ on $\kappa_n.$ By the closure of the poset together with the fact that we assumed $2^{\kappa_n}= \kappa_n^+,$ we have $U_n \in V.$ Let $\pi_n: V\to N_n$ be the elementary embedding corresponding to 
$U_n.$ Consider $\Coll^{N_n}(\kappa_n^{+\omega+2}, <\pi_n(\kappa_n)),$ this forcing has the $\pi_n(\kappa_n)$-chain condition in $N_n$ and $\pi(\kappa_n)$ is inaccessible. Therefore, there are $\pi_n(\kappa_n)$ many maximal antichains of this forcing which are in $N_n.$ 
On the other hand $\vert \pi_n(\kappa_n)\vert = \kappa_n^+$ and the forcing is $\kappa_n^+$-closed in $N_n.$ 
Therefore, one can inductively define in $V$ a generic filter $K_n$ for $\Coll^{N_n}(\kappa_n^{+\omega+2}, <\pi_n(\kappa_n))$ over $N_n$ by meeting each dense set in $N_n.$ 

\begin{remark} It is important that not only the ultrafilter $U_n$ is in $V,$ but also its choice is independent from the generic filter for $(\mathbb{S}\times \mathbb{C}_n)\ast \mathbb{R}\ast \mathbb{T}$. This is true, as this forcing is homogeneous, by Lemma~\ref{lem: iteration is homogeneous}.
\end{remark}

Now we are ready to define $\mathbb{P}$ in $V.$  

\begin{definition} Conditions of $\mathbb{P}$ are sequences of the form 
$$p= \langle \alpha_0, g_0, f_0,\ ...\ \alpha_{n-1}, g_{n-1}, f_{n-1}, A_n, g_n, F_n,\ ...\ \rangle$$
such that: 
\begin{enumerate}
\item every $\alpha_i$ is an inaccessible cardinal between $\kappa_{i-1}$ and $\kappa_i$ (with $\kappa_{-1}:= \omega$);
\item for $i<n,$ $g_i\in \Coll(\kappa_{i-1}^{++}, <\alpha_i);$
%\item $f_0\in \Coll(\alpha_0^{+\omega+2}, \kappa_0)$ and for $i\neq 0,$ 
\item $f_i\in \Coll(\alpha_i^{+\omega+2}, <\kappa_i);$
\item $A_j\in U_j$ and every element of $A_j$ is an inaccessible cardinal;
\item for $j\geq n,$ $g_j\in \Coll(\kappa_{j-1}^{++}, <\alpha)$ for the least $\alpha$ in $A_j$ (hence for every $\alpha\in A_j$);
\item $F_j$ is a function with domain $A_j$ such that $F_j(\alpha)\in \Coll(\alpha^{+\omega+2}, <\kappa_j)$ for every $\alpha\in A_j,$ and such that the equivalence class of $F_j$ as a member of the ultrapower $Ult(V, U_j)$ (denoted $[F_j]_{U_j}$) is in $K_j.$
\end{enumerate}

Given two conditions 
$$p= \langle \alpha_0^p, g_0^p, f_0^p,\ ...\ \alpha_{n-1}^p, g_{n-1}^p, f_{n-1}^p, A_n^p, g_n^p, F_n^p,\ ...\ \rangle$$
$$q= \langle \alpha_0^q, g_0^q, f_0^q,\ ...\ \alpha_{m-1}^q, g_{m-1}^q, f_{m-1}^q, A_m^q, g_m^q, F_m^q,\ ...\ \rangle$$
we say that $p\leq q$ if and only if, the following hold
\begin{enumerate}
\item $m\leq n$ and for $i<m,$ $\alpha_i^p= \alpha_i^q$ and $f_i^p\leq f_i^q;$
\item for every $i<\omega,$ $g_i^p\leq g_i^q;$
\item for $m\leq j< n,$ $\alpha_j^p\in A_i^q$ and $f_i^p\leq F_i^q(\alpha_j^p);$
\item for $j\geq n,$ $A_j^p\subseteq A_j^q$ and $F_j^p(\alpha)\leq F_j^q(\alpha)$ for all $\alpha\in A_j^p.$
\end{enumerate}

\end{definition}

We introduce some notations. Given a condition $p$ of $\mathbb{P}$ of the form  
$$p= \langle \alpha_0, g_0, f_0,\ ...\ \alpha_{n-1}, g_{n-1}, f_{n-1}, A_n, g_n, F_n,\ ...\ \rangle$$

we say that 
\begin{enumerate}
\item $n$ is the \emph{length} of $p,$ and we denote it $lg(p);$
\item the subsequence $\langle \alpha_0, g_0, f_0,\ ...\ \alpha_{n-1}, g_{n-1}, f_{n-1}\rangle$
is called the \emph{lower part of} $p$ or the \emph{stem of $p,$} denoted $\stem(p);$ 
\item $\langle \alpha_0,\ldots \alpha_{n-1} \rangle$ is the \emph{$\alpha$-part of $p$};
%\item $\langle f_0,\ldots f_{n-1} \rangle$ is the \emph{$f$-part of $p$};
\item $\langle g_i:\ i<\omega\rangle$ is the \emph{$g$-part of $p$} 
%and for $k<\omega,$ $\langle g_j:\ j\geq k\rangle$ is the \emph{$k$-upper-$g$-part of $p$};
%\item $\langle A_j:\ j\geq n\rangle$ is the \emph{$A$-part of $p$};
%\item $\langle F_j:\ j\geq n\rangle$ is the \emph{$F$-part of $p$};

\end{enumerate}

\ 

Given two conditions $p\leq q,$ such that $lg(p)= n$ and $lg(q)= m$ 
\begin{enumerate}
\item for $k\leq m,$ we write $p\restr k$ for the sequence $\langle \alpha_0, g_0, f_0,\ ...\ \alpha_{k-1}, g_{k-1}, f_{k-1}, g_k\rangle$ and we write $p\leq_k q$ when $p\leq q,$ $lg(p)= lg(q)$ and $p\restr k= q\restr k.$  
\item for $k\leq n,$ we say that $p$ is a \emph{$k$-direct extension} of $q$ if \begin{enumerate}
\item $f_i^p= f_i^q$ for $m>i\geq k$
\item $g_i^p= g_i^q$ for every $i>k$
\item $A_j^p= A_j^q$ for every $j\leq n$ except when $k= m= n,$ in that case 
$A_n= \{\alpha\in A_n^q ;\ g_n^p\in \Coll(\kappa_{n-1}^{++}, <\alpha) \}$
\item $F_j^p(\alpha)= F_j^q(\alpha)$ for every $j\leq n$ and $\alpha\in A_j^p$
\item $f_i^p= F_i^q(\alpha)$ for $i\geq m, k$ 
\end{enumerate} 
(informally $p$ is a $k$-direct extension of $q$ if $p$ does not add information on the collapses above $\alpha_k.$)
We say that $p$ is a \emph{direct extension} of $q$ if it is a $0$-direct extension and $g_0^p= g_0^q.$
\item for $k\leq n,$ the \emph{$k$-interpolation} of $q$ and $p,$ denoted $Int(k, q, p)$ is the unique condition $s$ such that $s$ is a $k$-direct extension of $q$ and $s\leq_k p.$
\end{enumerate}

We state some important facts about $\mathbb{P}.$

\begin{lemma}\label{lemma: Prikry property for P} (Magidor Shelah \cite[Lemma 3, p. 791]{MagidorShelah})
$\mathbb{P}$ has the \emph{Prikry property}, namely for every open subset $D$ of $\mathbb{P},$ for every condition $p\in \mathbb{P}$ and every $k\leq lg(p),$ there exists a condition $q\leq_k p$ such that 
\begin{enumerate}
\item for every condition $q^*\leq q$ in $D,$ we have $Int(k, q, q^*)\in D$ 
\item given $q^*\leq q$ in $D,$ for every condition $q^{**}\leq q$ with the same length of $q^*$ such that $q^{**}\restr k= q^*\restr k,$ we have $q^{**}\in D$ 
\item for $q^*$ and $q^{**}$ as above, if $D$ is the set of conditions deciding a given statement $\varphi,$ then we can assume that $q^*\force \varphi$ if and only if $q^{**}\force \varphi.$
\end{enumerate}
\end{lemma}

\begin{lemma}(Magidor Shelah \cite[p. 789]{MagidorShelah}) In  $V^{\mathbb{P}},$ every $\kappa_n$ is equal to $\aleph_{\omega(n+1)+3},$  $\lambda= \aleph_{\omega^2}$ and 
$\lambda^+=\aleph_{\omega^2+1}.$
\end{lemma}

One can easily see that the poset $\mathbb{P}^*$ projects to $\mathbb{C}_{fin}.$ So, in $V^{\mathbb{C}_{fin}\ast \mathbb{S}\ast \mathbb{R}},$ we define $\mathbb{P^*}$ as the quotient forcing $\mathbb{P}/\mathbb{C}_{fin},$ namely the set of conditions $p\in \mathbb{P}$ whose $g$-part belongs to the generic for $\mathbb{C}_{fin};$ $\mathbb{P}^*$ is ordered as a subposet of $\mathbb{P}.$ Our final model is 
$$V^{\mathbb{C}_{fin}\ast \mathbb{S}\ast \mathbb{R}\ast \mathbb{P}^*}$$

We conclude this section by showing that $\mathbb{P}^*$ does not collapse the relevant cardinals. 

\begin{lemma} $\mathbb{P}^*$ is $\lambda^+$-c.c. in $V^{\mathbb{C}_{fin}\ast \mathbb{S}\ast \mathbb{R}\ast \mathbb{T}}$ and in particular in $V^{\mathbb{C}_{fin}\ast \mathbb{S}\ast \mathbb{R}}.$
\end{lemma}

\begin{proof} In \cite[Proposition 4.5]{FontanellaMagidor} it is proven that $\mathbb{P}^*$ is $\lambda^+$-c.c. in $V^{\mathbb{C}_{fin}}$ (the statement of \cite[Proposition 4.5]{FontanellaMagidor} is actually stronger). In Corollary~\ref{coroll: directed closure Cfin} we showed that 
$\mathbb{S}\ast\mathbb{R}\ast\mathbb{T}$ is $\lambda^+$-closed in $V^{\mathbb{C}_{fin}}.$ It follows that $\mathbb{P}^*$ remains $\lambda^+$-c.c. even in $V^{\mathbb{C}_{fin}\ast \mathbb{S}\ast \mathbb{R}\ast \mathbb{T}}.$ 

\end{proof}

%\begin{proof} There are at most $\lambda$ possible stems for conditions of $\mathbb{P}^*,$ and there are at most $\lambda$ possible $g$-parts for conditions of $\mathbb{P}^*.$ Let $\{s_\alpha;\alpha<\lambda\}$ be an enumeration of the possible stems, and let $\{\vec{g}_\beta;\ \beta<\lambda\}$ be an enumeration of the possible $g$-parts of a condition in $\mathbb{P}^*.$ For each $\alpha, \beta<\lambda,$ we fix a condition $p_{\alpha,\beta}$ with stem $s_\alpha$ and $g$-part $\vec{g}_\beta.$
%If $A$ is a maximal antichain of $\mathbb{P}^*,$ then for every condition of $p$ there are $\alpha,\beta<\lambda$ such that the stem of $p$ is $s_\alpha$ and the $g$-part of $p$ is $\vec{g}_\alpha.$ Then $p$ is compatible with $p_{\alpha,\beta}$ because two conditions with the same stem and the same $g$-part must be compatible in $\mathbb{P}^*.$ It follows that $A$ has size at most $\lambda.$ \end{proof}

\begin{lemma} In 
$V^{\mathbb{C}_{fin}\ast \mathbb{S}\ast \mathbb{R}\ast \mathbb{P}^*},$ every $\kappa_n$ is equal to $\aleph_{\omega(n+1)+3},$  $\lambda= \aleph_{\omega^2}$ and 
$\lambda^+=\aleph_{\omega^2+1}.$
\end{lemma}

\begin{proof} First we point out that 
$$V^{\mathbb{C}_{fin}\ast \mathbb{S}\ast \mathbb{R}\ast \mathbb{P}^*}=
V^{\mathbb{C}_{fin}\ast \mathbb{P}^*\ast \mathbb{S}\ast \mathbb{R}}
= V^{\mathbb{P}\ast \mathbb{S}\ast \mathbb{R}}$$ 
Magidor and Shelah showed that in $V^{\mathbb{P}},$ every $\kappa_n$ is equal to $\aleph_{\omega(n+1)+3},$  $\lambda= \aleph_{\omega^2}$ and 
$\lambda^+=\aleph_{\omega^2+1}.$ We proved that $\mathbb{S}\ast \mathbb{R}\ast \mathbb{T}$ contains a $\lambda^+$-closed dense subset, in particular, $\mathbb{S}\ast \mathbb{R}$ does not collapse cardinals below $\lambda^+,$ thus the conclusion follows. \end{proof}

\section{Square and Delta reflection at \texorpdfstring{$\aleph_{\omega^2+1}$}{aleph omega 2}}

In this section we prove Theorem~\ref{thm: main theorem}, namely we want to prove that both $\square(\aleph_{\omega^2+1})$ and $\Delta_{\aleph_{\omega^2}, \aleph_{\omega^2+1}}$ hold in the model 
$V^{\mathbb{C}_{fin}\ast \mathbb{S}\ast \mathbb{R}\ast \mathbb{P}^*}.$ 
We already pointed out that $\lambda^+$ is $\aleph_{\omega^2+1}$ in this model, therefore we want to show that both $\Delta_{\lambda, \lambda^+}$ and $\square(\lambda^+)$ hold. First we argue that $\square(\lambda^+)$ holds. By Corollary~\ref{coroll: existence of square sequence}, we know that a square sequence $\mathcal{C}$ with no threads exists in 
$V^{\mathbb{C}_{fin}\ast \mathbb{S}\ast \mathbb{R}}.$ Using the fact that $\mathbb{P}^*$ is $\lambda^+$-c.c.\ we can apply Lemma~\ref{lemma: preserving the square sequence} to shows that $\mathbb{P}^*$ does not add a thread. Thus $\square(\lambda^+)$ holds in $V^{\mathbb{C}_{fin}\ast \mathbb{S}\ast \mathbb{R}\ast \mathbb{P}^*}.$   

Now, we want to prove that the Delta reflection holds at $\aleph_{\omega^2+1}$ in this model. The proof is along the lines of \cite{MagidorShelah}. We are going to prove that the Delta reflection holds at $\lambda^+$ in a $\mathbb{T}$-generic extension of this model. Lemma \ref{lem: c.c. preserves generically stationary preservation} implies that $\mathbb{T}$ generically preserves stationary sets also after forcing with $\mathbb{P}^*.$
Therefore, we can apply Lemma \ref{lem: avoiding T}, hence if $V^{\mathbb{C}_{fin}\ast \mathbb{S}\ast \mathbb{R}\ast \mathbb{P}^*\ast \mathbb{T}}$ is a model of the Delta reflection at $\aleph_{\omega^2+1},$ then so is $V^{\mathbb{C}_{fin}\ast \mathbb{S}\ast \mathbb{R}\ast \mathbb{P}^*}.$ Clearly, we have 
$$V^{\mathbb{C}_{fin}\ast \mathbb{S}\ast \mathbb{R}\ast \mathbb{P}^*\ast \mathbb{T}}= V^{\mathbb{C}_{fin}\ast \mathbb{S}\ast \mathbb{R}\ast \mathbb{T}\ast \mathbb{P}^*}.$$ 

For convenience, we fix generic filters: 
\begin{itemize}
\item $G_{\mathbb{C}_{fin}}\subseteq\mathbb{C}_{fin}$ is a $V$-generic filter;
\item $G_\mathbb{S}\subseteq \mathbb{S}$ is a $V[G_{\mathbb{C}_{fin}}]$-generic filter;  
 \item $G_{\mathbb{R}}\subseteq\mathbb{R}$ is a $V[(G_\mathbb{S}\times G_{\mathbb{C}_{fin}})]$-generic filter; 
\item $G_{\mathbb{T}}\subseteq\mathbb{T}$ is a $V[(G_\mathbb{S}\times G_{\mathbb{C}_{fin}})\ast G_\mathbb{R}]$-generic filter. 
\end{itemize}

We start by working in $\bar{W}= V[G_\mathbb{S}\times G_{\mathbb{C}_{fin}}][G_\mathbb{R}][G_\mathbb{T}].$
Let $p\in \mathbb{P}^*$ and let $\dot{S}$ and $\dot{A}$ be $\mathbb{P}^*$-names such that: 
$$\begin{array}{rl}
p\force & ``\dot{A}\textrm{ is an algebra on $\lambda^+$ 
with $\mu<\lambda^+$ operations}\\
& \textrm{ and }\dot{S}\subseteq \lambda^+\textrm{ is  a stationary set }''\\ 
\end{array}$$

Without loss of generality, the cofinalities of the elements of $\dot{S}$ are bounded below $\lambda$. 
Let $l$ be large enough so that $\mu <\kappa_l$ and $p\force \forall \alpha\in \dot{S}\ (cof(\alpha)<\check{\kappa_l}).$ We can assume without loss of generality that $lg(p)>l;$ let $n$ be the length of $p.$ 

\begin{remark}\label{rmk: one} We argue for a moment in $\bar{W}^{\mathbb{P}^*}.$ 
For every $\alpha\in \dot{S}^{G_{\mathbb{P}^*}}$ there is a condition $q_\alpha$ in the generic for $\mathbb{P}^*$ such that $q_\alpha\force \alpha\in \dot{S}.$ There are less than $\lambda^+$ many possible stems for the $q_\alpha$'s, therefore the conditions $q_\alpha$ have a fixed stem $s$ for stationary many $\alpha$ in $\dot{S}^{G_{\mathbb{P}^*}}.$ Without loss of generality we can assume that the stem of $p$ extends $s,$ hence 
$$ p \force_{\mathbb{P}^*}^{\bar{W}}\ \dot{S^{\prime}}:= \{\alpha<\lambda^+;\ \exists q\in \dot{G}_{\mathbb{P}^*}\ \stem(q)= \stem(p)\land q\force \alpha\in \dot{S}\}\textrm{ is stationary}$$     
\end{remark}

It follows that if $q\leq p$ and $q\force \alpha\in \dot{S^{\prime}}$ for some $\alpha,$ then there is an $n$-length preserving extension $q^*$ of $p$ such that $q\leq q^*\leq p$ and $q^*\force \alpha\in \dot{S^{\prime}}.$
% Indeed, if we let $q^*$ be $stem(p)\smallfrown upper(q)$ ($q$ may be of %length $>n$), then $q\leq q^*\leq p$ and it is easy to check that 
%$q^*\force \alpha\in \dot{S}$ because $q^*$ forces that  
%there exists $r$ in the generic with the same stem as $p$
%and such that $r\force \alpha\in\dot{S},$ then if $s\leq r, q^*,$ we have that %$s\force \alpha\in \dot{S},$ this proves that $q^*$ it self forces 
%$\alpha\in \dot{S}$ 

Now, we force with $\mathbb{C}_n/G_{\mathbb{C}_{fin}}$ to get a generic filter $G_n$ for $\mathbb{C}_n$ such that $G_n$ projects to $G_{\mathbb{C}_{fin}}$ and $\langle g_i^p:\ i\geq n\rangle\in \mathbb{C}_n.$ Let $\bar{W}_n:= V[G_\mathbb{S}\times G_{\mathbb{C}_n}][G_\mathbb{R}][G_{\mathbb{T}}],$ the main reason for moving to this model is that here $\kappa_n$ is supercompact. The main idea of the proof is the following: we are going to define in $\bar{W}_n$ a ``fake'' version of the stationary set $\dot{S}$ and a ``fake'' version of the algebra $\dot{A}$ (formally the fake version of $\dot{A}$ will be defined in a further extension of $\bar{W}_n,$ but we will deal with that later), then we will use the supercompactness of $\kappa_n$ to find a subalgebra of the fake algebra of order type $<\kappa_n$ such that the fake version of the stationary set reflects to such subalgebra. Finally, we will show that the fake version of $\dot{A}$ and the fake version of $\dot{S}$ are close enough to $\dot{S}$ and $\dot{A}$ so that we will be able to extend $p$ to a condition that forces the conclusion of the Delta reflection for $\dot{S}$ and $\dot{A}.$   

In $\bar{W}_n$ we define a subposet $\mathbb{P}_n^*$ of $\mathbb{P}^*.$ 
$$\mathbb{P}_n^*:= \{q\in \mathbb{P}^*;\ lg(q)\geq n \textrm{ and }\langle g_i^q:\ n\leq i<\omega\rangle\in G_n\}.$$
%(technically the sequence $\langle g_i^q:\ lg(q)\leq i<\omega\rangle$ is a %condition in $\mathbb{C}_{lg(q)}$ and the length of $q$ may be larger than %$n,$ however there is a natural projection of $\mathbb{C}_n$ to %$\mathbb{C}_{lg(q)}$ thus we can consider that $G_n$ naturally induce a %generic filter for $\mathbb{C}_{lg(q)}$). 

     \begin{lemma} The forcing $(\mathbb{C}_n/G_{\mathbb{C}_{fin}})\ast \mathbb{P}_n^*$ projects to $\mathbb{P}^*$
 \end{lemma} 

\begin{proof} We define a projection $\pi$ as follows. 
Let $\langle c, r\rangle$ be a condition in $(\mathbb{C}_n/G_{\mathbb{C}_{fin}})\ast \mathbb{P}_n^*.$ Let $c=\langle c_i;\ i\geq n\rangle$ and let $m$ be the length of $r$ (we have $m\geq n$ by definition of $\mathbb{P}_n^*$). For every $i<n,$ let $c_i=\emptyset,$ then, since $c\in \mathbb{C}_n/G_{\mathbb{C}_{fin}},$ it means that 
the equivalence class of $\langle c_i:\ i<\omega \rangle$ is in $G_{\mathbb{C}_{fin}}.$ Moreover, since $c\force r\in \mathbb{P}_n^*,$ we have $\langle c_i;\ i\geq n\rangle\leq \langle g_i^r;\ i\geq n\rangle.$ We define $\pi(\langle c,r\rangle)$ as the condition $q$ of length $m$ obtained from $r$ by replacing the $g$-part of $r$ with the sequence $\langle g_i;\ i<\omega\rangle,$ where 
\begin{itemize}
\item $g_i= g_i^r,$ for $i<n;$
\item $g_i= c_i,$ for $i\geq m;$ 
\item if $n<m,$ then we let $g_i= c_i\restr \alpha_i,$ for $n\leq i<m.$ 
\end{itemize}
Then $q$ is clearly in $\mathbb{P}^*$ because its $g$-part $\langle g_i;\ i<\omega\rangle$ is equivalent to $\langle c_i:\ i<\omega\rangle,$ and the equivalence class of $\langle c_i:\ i<\omega\rangle$ is in $G_{\mathbb{C}_{fin}}.$ It is not difficult to verify that $\pi$ is a projection. \end{proof}

It follows, by this lemma, that if we force with $\mathbb{P}_n^*$ over $\bar{W}_n,$ we introduce a generic filter for $\mathbb{P}^*$ over $\bar{W},$ thus we can consider $\bar{W}^{\mathbb{P}^*}$ to be a submodel of 
$\bar{W}_n^{\mathbb{P}_n^*}.$ 

Now, we prove some useful properties of $\mathbb{P}_n^*.$

\begin{lemma}\label{lemma: closure of the k-ordering for Pn} Let $\kappa=\kappa_l$ for some $l<\omega,$ and let $\langle p_\delta;\ \delta<\eta\rangle$ be a $\leq_k$-decreasing sequence of conditions in $\mathbb{P}_n^*$ all of length $l<\omega.$ Suppose $\eta\leq \alpha_k^p$ if $k<l$ and $\eta<\kappa$ if $k=l.$ Then the sequence has a lower bound in $\mathbb{P}_n^*.$   
\end{lemma}

\begin{proof} We define the lower bound $q$ as follows. The conditions of the sequence have the same $\alpha$-sequence 
$\langle \alpha_i;\ i< l\rangle.$ We let $q$ be 
$$\langle \alpha_0, g_0, f_0, \ldots \alpha_{l-1}, g_{l-1}, f_{l-1}, A_l, g_l, F_l, \ldots \rangle,$$

    where for $i<l,$ $f_i$ is $\bigcup_{\delta<\eta}f_i^{p_\delta}$ (this is well defined by the closure of $\Coll(\alpha_i^{+\omega+2}, <\kappa_i)$); similarly, for every $j,$ $g_j$ is $\bigcup_{\delta<\eta}g_j^{p_\delta}$ and $A_j$ is $\bigcap_{\delta<\eta} A_j^{p_\delta};$ finally for $\alpha\in A_j,$ we let $F_j(\alpha)=\bigcup_{\delta<\eta}F_j^{p_\delta}(\alpha)$ -- the equivalence class of $F_j$ is in $K_j$ because $N_j$ is closed under $\eta$-sequences and $K_j$ is generic for a forcing notion considered to be $\eta$-closed by $N_j.$

Now, we check that $q$ is a condition in $\mathbb{P}_n^*.$ Each condition $p_\delta$ is in $\mathbb{P}_n^*,$ hence we have that for each $\delta<\eta,$ the sequence 
$\langle g_j^{p_\delta};\ j\geq n\rangle$ is in $G_n.$ Each $g_j^{q}$ is defined as $\bigcup_{\delta<\eta} g_j^{p_\delta},$ thus the sequence 
$\langle g_j^q;\ j\geq n\rangle$ belongs to the generic $G_n,$ as required. \end{proof}

\begin{lemma}\label{lemma: Prikry property for Pn} 
$\mathbb{P}_n^*$ has the \emph{Prikry property}, namely for every open subset $D$ of $\mathbb{P}_n^*,$ for every condition $p\in \mathbb{P}_n^*$ and every $k\leq lg(p),$ there exists a condition $q\leq_k p$ such that 
\begin{enumerate}
\item for every condition $q^*\leq q$ in $D,$ we have $Int(k, q, q^*)\in D$ 
\item given $q^*\leq q$ in $D,$ for every condition $q^{**}\leq q$ with the same length of $q^*$ such that $q^{**}\restr k= q^*\restr k,$ we have $q^{**}\in D$ 
\item for $q^*$ and $q^{**}$ as above, if $D$ is the set of conditions deciding a given statement $\varphi,$ then we can assume that $q^*\force \varphi$ if and only if $q^{**}\force \varphi.$
\end{enumerate}
\end{lemma}

For the proof of this lemma we refer to \cite[Lemma 3, p. 791]{MagidorShelah} where Magidor and Shelah prove the Prikry property for $\mathbb{P}.$ The same arguments apply to prove the Prikry property for $\mathbb{P}_n^*,$ because there are only two essential ingredients in that proof. First, every $\leq_k$-decreasing sequence of less than $\kappa$ many conditions of the same length must have a lower bound, which is what we just proved in Lemma~\ref{lemma: closure of the k-ordering for Pn} for our forcing $\mathbb{P}_n^*.$ Second, the $A$-parts and the $F$-parts must be closed by diagonal intersections, which follows from the normality of the ultrafilters $U_i,$ so it is true also of our forcing $\mathbb{P}_n^*.$

Any condition $q$ satisfying Lemma~\ref{lemma: Prikry property for Pn} is said to be in $D$ \emph{modulo $k$-direct extensions}. 

\begin{remark} 
Note that if $D$ is the dense set of conditions deciding the truth value of some statement $\varphi,$ then for $q, q^*$ as in the lemma we have that, if $q^*\leq q$ in $D,$ then for every condition $r\leq q,$ if  $r\restr k= q^*\restr k,$ then $r$ decides $\varphi$ the same way $q^*$ does. This means that for deciding $\varphi$ below $q$ we only need to change $q\restr k.$
\end{remark}

Now, we define in $\bar{W}_n$ a ``fake version'' of the stationary set $\dot{S}$ that we call $S^*,$  

$$S^*:=\{\alpha\in \lambda^+;\ \exists q\in\mathbb{P}_n^*\ (\stem(q)=\stem(p) \land q\force \alpha\in \dot{S}')\}.$$

We show that $S^*$ is stationary in $\bar{W}_n.$ Indeed, if $C\subseteq \lambda^+$ is a club in $\bar{W}_n,$ then by the $\lambda^+$-chain condition of $\mathbb{C}_n/G_{\mathbb{C}_{fin}}$ there is a club $D\subseteq C$ which is in $\bar{W}.$ In particular $D$ is in $\bar{W}^{\mathbb{P}^*}$ and, by Remark~\ref{rmk: one}, the realization of $\dot{S^{\prime}}$ in this model is a stationary set. Thus $D$ has non empty intersection with $\dot{S^{\prime}},$ which is clearly a subset of $S^*.$  

Note also that all the ordinals of $\mathbb{S}^*$ have cofinality less than $\kappa_n.$  

The next step is to define a ``fake'' version of the algebra in a generic extension of $\bar{W}_n$ via the product 
$$\mathbb{Z}:= \prod_{i<n} \Coll(\alpha_i^{+\omega+2}, <\kappa_i)\times \prod_{i<n} \Coll(\kappa_{i-1}^{++}, <\alpha_i)\times \Coll(\kappa_{n-1}^{++}, <\kappa_n),$$
where $\langle \alpha_0, \ldots , \alpha_{n-1}\rangle$ is the $\alpha$-sequence of $p.$

Let $G_{\mathbb{Z}}$ be a generic filter for $\mathbb{Z}$ containing the stem of $p.$ We define $\mathbb{P}^{**}$ as the set of all conditions $q\in \mathbb{P}_n^*$ of length $n$ such that $\stem(q)\in G_{\mathbb{Z}}.$

\begin{remark}\label{rmk: conditions of P** are pw compatible} Note that every two conditions of $\mathbb{P}^{**}$ are compatible in $\mathbb{P}_n^{*}$, indeed if $q, r\in \mathbb{P}^{**},$ then the stems are compatible since they belong to the generic filter $G_{\mathbb{Z}};$ the $g$-parts are compatible (in $\mathbb{C}_n$) since they belong to the generic filter $G_n;$ the $A$-parts are compatible since for every $j\geq n,$ we have $A_j^q\cap A_j^r\in U_j;$ finally the $F$-parts are compatible since the equivalence class of $F_j^q$ and $F_j^r$ modulo $U_j$ belong to the generic ultrafilter $K_j,$ thus 
$\{\alpha;\ F_j^q(\alpha)\textrm{ is compatible with }F_j^r(\alpha)\}\in U_j.$ 
\end{remark}

%By Remark~\ref{rmk: conditions of P** are pw compatible}, 
It follows that for any statement $\varphi$ in the forcing language for $\mathbb{P}_n^*$ (over $\bar{W}_n$) any two conditions in $\mathbb{P}^{**}$ which decide $\varphi$ assign to $\varphi$ the same truth value.

%\begin{lemma}\label{lemma: deciding statements within P double star} For every statement $\varphi$ in the forcing language for $\mathbb{P}_n^*$ (over $W_n$) there exists $r\leq p$ in $\mathbb{P}^{**}$ such that $r$ decides $\varphi.$
%\end{lemma}

%\begin{proof} Given a statement $\varphi$ in the forcing language for $\mathbb{P}_n^*.$ Suppose by contradiction that there is $(s, h)\in \mathbb{Z}\times \mathbb{C}_n$ such that $s\leq stem(p)$ and $h= \langle h_i:\ i\geq n\rangle\leq \langle g_i^p:\ i\geq n\rangle$ and $(s, h)$ forces that the lemma fails. Let $p^*$ be the condition obtained from $p$ by replacing $\stem(p)$ with $s,$ and by replacing $\langle g_i^p:\ i\geq n\rangle$ by $h.$ Then, $(s, h)$ forces that $p^*\in \mathbb{P}_n^*$ and cearly $p^*\leq p.$ By Lemma~\ref{lemma: Prikry property for Pn}, we know that $\mathbb{P}_n^*$ has the Prikry property, so we can find a condition $q\leq_n p^*$ in $\mathbb{P}_n^*$ that decides $\varphi$ modulo $n$-direct extensions (we recall that the definition of $\leq_n$ implies that $lq(q)=lg(p)$). This means that there exists $q^*\leq q$ in $\mathbb{P}_n^*$ which is an $n$-direct extension of $q$ and decides $\varphi.$ By the remarks after Lemma~\ref{lemma: Prikry property for Pn}, we can assume that $lg(q^*)= n.$ Let $t^*$ be the \stem of $q^*$ and let $g^*=\langle g_i^{q^*}:\ i\geq n\rangle.$ Then, clearly $(t^*, h^*)$ forces that $q^*$ is in $\mathbb{P}^{**}$ and $q^*$ decides $\varphi,$ but $(t^*, h^*)\leq (s, h),$ contradicting the fact that $(s, h)$ forces the lemma to fail. \end{proof}

\begin{lemma}\label{lemma: Prikry property for P double star} Let $D\subseteq \mathbb{P}_n^*$ be a dense set (in $\bar{W}_n$), then there is $r\leq p$ in $\mathbb{P}^{**}$ such that $r$ is in $D$ modulo direct extensions, and there exists $m$ such that any direct extension of $r$ of length $m$ is in $D.$ 
\end{lemma}

 \begin{proof} Fix a dense set $D\subseteq \mathbb{P}_n^*.$ Suppose by contradiction that there is $s\in \mathbb{Z}$ such that $s\leq \stem(p)$ and $s$ forces that the lemma fails. Let $p^*$ be the condition obtained from $p$ by replacing $\stem(p)$ with $s.$ Then, $s$ forces that $p^*\in \mathbb{P}^{**}$ and clearly $p^*\leq p.$ By Lemma~\ref{lemma: Prikry property for Pn}, we know that $\mathbb{P}_n^*$ has the Prikry property, so we can find a condition $q\leq_n p^*$ in $\mathbb{P}_n^*$ that is in $D$ modulo $n$-direct extensions. Let $r^*\leq q$ be in $D,$ and let $m\geq n$ be its length. $q$ has the property that every extension $q^*$ of $q$ of length $m$ such that $q^*\restr n= r^*\restr n$ is in $D.$
 We let $r$ be the condition obtained from $q$ by replacing $\stem(q)$ with $\stem(r^*)\restr n.$ If $t= \stem(r),$ then $t\leq s$ and $t$ forces that $r$ is in $\mathbb{P}^{**},$ moreover $r$ is in $D$ modulo direct extensions and every direct extension of $r$ of length $m$ is in $D,$ contradicting the fact that $s$ forces the lemma to fail. \end{proof}
 
Similar arguments as for the proof of Lemma~\ref{lemma: Prikry property for P double star} imply the following lemma. 

\begin{lemma}\label{lemma: deciding statements within P double star} For every statement $\varphi$ in the forcing language for $\mathbb{P}_n^*$ (over $\bar{W}_n$) there exists $r\leq p$ in $\mathbb{P}^{**}$ such that $r$ decides $\varphi$ (in the sense of the forcing $\mathbb{P}_n^*$)
\end{lemma}

 The fake version of the algebra is defined as follows. We can assume that for each $m<\omega,$ the algebra $\dot{A}$ has $\mu$ many operations of arity $m$ 
 (so we already know in $\bar{W},$ for each $l<\omega,$ how many $l$-ary operations are in $\dot{A}$). We fix $\mathbb{P}^*$-names for the Skolem functions $\langle t_\rho\rangle_{\rho\in \mu}$ appropriate to the language of $\dot{A}.$ 
 
 In $\bar{W}_n[G_{\mathbb{Z}}],$ we let $A^*$ be set of sequences 
 $(\rho, \alpha_1,\ldots, \alpha_l)$ such that $\rho\in \mu,$ the Skolem function $t_\rho$ is of arity $l$ and $\alpha_1,\ldots,\alpha_l$ is a finite sequence of ordinals in $\lambda^+.$ We define an equivalence relation $\sim^*$ on $A^*$ by letting 
 $$\begin{matrix}
 (\rho, \alpha_1,\ldots, \alpha_l)\sim^* (\eta, \beta_1, \ldots ,\beta_m)\\ \ \iff \\\exists q\in \mathbb{P}^{**}\ (q\force_{\mathbb{P}_n^* }\dot{t}_\rho(\alpha_1, \ldots, \alpha_l) = \dot{t}_\eta(\beta_1, \ldots, \beta_m))
 \end{matrix}$$
   
Note that, from what we proved of $\mathbb{P}^{**},$ either there exists $q\in\mathbb{P}^{**}$ forcing that $\dot{t}_\rho(\alpha_1, \ldots, \alpha_l) = \dot{t}_\eta(\beta_1, \ldots, \beta_m),$ or there exists $q\in \mathbb{P}^{**}$ forcing that $\dot{t}_\rho(\alpha_1, \ldots, \alpha_l) \neq \dot{t}_\eta(\beta_1, \ldots, \beta_m),$ and the two possibilities are exclusive. We also define an ordering $<^*$ on $A^*$ by letting 

$$\begin{matrix}
(\rho, \alpha_1,\ldots, \alpha_l)<^* (\eta, \beta_1, \ldots ,\beta_m) \\ \iff \\ \exists q\in \mathbb{P}^{**}\ (q\force_{\mathbb{P}_n^*} \dot{t}_\rho(\alpha_1, \ldots, \alpha_l) < \dot{t}_\eta(\beta_1, \ldots, \beta_m))\end{matrix}$$

It is not difficult to see that $<^*$ is a well order on $A^*.$ Moreover, the order type of $A^*$ under the order $<^*$ is $\lambda^+;$ for a proof of this fact we refer to \cite[Lemma 10 and Lemma 11, p. 801]{MagidorShelah}

%\begin{lemma} The following holds in $W_n[G_{\mathbb{Z}}]:$
%\begin{enumerate}
%\item\label{item: one} $S^*$ is stationary
%\item\label{item: two} every ordinal in $S^*$ has cofinality $<\kappa_n$ 
%\end{enumerate}
%\end{lemma}

%\begin{proof} 
%\noindent \begin{enumerate}
%\item[(\ref{item: one})] To show that $S^*$ is still stationary in $W_n[G_{\mathbb{Z}}],$ 
%suppose $C\subseteq \lambda^+$ is a club in $W_n[G_{\mathbb{Z}}].$ $\mathbb{Z}$ is $\lambda^+$-c.c.  in $W_n,$ hence there exists a club $D\subseteq C$ such that $D$ belongs to $W_n.$ Since $S^*$ was stationary in this model, it has non empty intersection with $D$ and in particular with $C.$ 

%\item[(\ref{item: two})] For every $\alpha\in S^*$ there is a condition $q_\alpha\leq p$ in $\mathbb{P}^*$ such that $q_\alpha\Vdash\alpha\in\dot{S},$ moreover we assumed that $p\Vdash \alpha \in \dot{S}\implies \cof \alpha < \kappa_n.$ Since $\mathbb{P}^*$ does not collapse $\kappa_n,$ it follows that every $\alpha\in S^*$ has cofinality $<\kappa_n$ in $W.$ $\mathbb{C}_n/G_{\mathbb{C}_{fin}}$ does not change cofinalities below $\kappa_n^{++}$ either, thus $\alpha$ has cofinality $<\kappa_n$ in $W_n.$ Since $\mathbb{Z}$ is $\kappa_n$-c.c., the lemma follows.

%\end{enumerate}
%\end{proof}

Now we step back to $\bar{W}_n.$ 
%Recall that the supercompactness of $\kappa_n$ is indestructible by $\kappa_n^+$-directed closed forcings. It follows that $\kappa_n$ is still supercompact in $\bar{W}_n[G_{\mathbb{T}}].$ 
Recall that $U_n^*$ is a normal ultrafilter on $\mathcal{P}_{\kappa_n}(\lambda^+)$ and $U_n$ corresponds to its projection to a normal ultrafilter on $\kappa_n.$ 
We let $j$ be the $\lambda^+$-supercompact embedding corresponding to $U_n^*.$ Let $\dot{A}^*$ be a $\mathbb{Z}$-name for the fake algebra $A^*.$ Fix $\theta$ large enough such that $\lambda^+,$ $S^*,$ $\mathbb{P}^*,$ $p,$ $\mathbb{Z},$ $\dot{A}^*$ are in $H_\theta.$ Let $\mathcal{H}$ be the structure $\langle H_\theta, \lambda^+, S^*, \mathbb{P}^*, p, \mathbb{Z}, \dot{A}^* \rangle $

By arguing as in the proof of Proposition \ref{prop: reflection from spc} we can find a set $X\in \mathbb{P}_{\kappa_n}(\lambda^+)$ such that the following holds: 
\begin{itemize}
\item[(1)]\label{item: primo} $S^*\cap X$ is stationary in $\sup(X),$
\item[(2)]\label{item: secondo} $o.t.(X)= (X\cap \kappa_n)^{+\omega+1}$ (hence $o.t.(X)<\kappa_n$)
%\item $X\cap \kappa_n\in A$ for every $A\in U_n$ (see Remark \ref{rmk: pippo})
\item[(3)]\label{item: terzo} If $N$ is $Sk_{\mathcal{H}}(X)$ (i.e. the Skolem closure of $X$ to get an elementary substructure of $\mathcal{H}$ of the same size as $\vert X\vert$), then $N\cap (\lambda^+)^{<\omega}= (X)^{<\omega}$ 
\item[(4)]\label{item: quarto} $X\cap \kappa_n\in E$ for every $E\in U_n\cap N.$
\end{itemize}

Indeed, if we consider $X^*:= j[\lambda^+],$ then 
$o.t.(X^*)= \lambda^+=
\kappa_n^{+\omega+1}= (X^*\cap \kappa_n)^{+\omega+1}$ 
and we can show as in Proposition \ref{prop: reflection from spc} that 
$j(S^*)\cap X^*$ is stationary in $\sup(X^*).$ Moreover, if $N^*= Sk_{j(H)}(X^*),$ then for every $E^*\in j(U_n)\cap N^*= j[U_n]\cap N^*$ there exists 
$E\in U_n$ such that $j(E)=E^*$ and $j[\lambda^+]\cap j(\kappa_n)=\kappa_n\in j(E).$ Therefore, the existence of such a set $X$ follows by elementarity of the $j.$  

Note that we can also assume that $\mu\subseteq X.$
We defined $X$ in $\bar{W}_n,$ however $\mathbb{C}_n/G_{\mathbb{C}_{fin}}$ is $\kappa_n$-distributive, thus $X$ is in $\bar{W}$.

Item~\ref{item: terzo} above implies in particular that $X\cap \kappa_n$ belongs to every $A_n^r$ with $r\in \mathbb{P}_n^*\cap N,$ hence $X\cap \kappa_n$ is an inaccessible cardinal below $\kappa_n$ (by the definition of $\mathbb{P},$ every set $A_n^r$ contains only inaccessible cardinals).

\begin{lemma} In $\bar{W},$ there exists a condition $q\leq p$ in $\mathbb{P}_n^*$ of length $n+1$ such that the 
$n$-th term of the $\alpha$-sequence of $q$ is precisely $X\cap \kappa_n$ and such that $q$ extends every trivial extension of $p$ (every extension with the same stem) which is in $\mathbb{P}_n^*\cap N.$
\end{lemma}

\begin{proof} We let 
$q= \stem(p)\smallfrown \langle \alpha_n, g_n, f_n, A_{n+1}, g_{n+1}, F_{n+1},\ldots \rangle$
where $\alpha_n= X\cap \kappa_n,$ and the other components are defined as follows: 
\begin{enumerate}
\item[(1)] each $g_i$ is the union of all the $g_i^{r}$ for $r\in \mathbb{P}_n^*\cap N$
\item[(2)] $A_i$ is the intersection of all the $A_i^{r}$ for $r\in \mathbb{P}_n^*\cap N$
\item[(3)] $f_n$ is the union of all the $F_n^{r}(\alpha)$ for $r\in \mathbb{P}_n^*\cap N$
\item[(4)] $F_i$ is the function on $A_i$ defined by $F_i(x)= \bigcup\{ F_i^{r}(x);\ r\in \mathbb{P}_n^*\cap N \land [F_i]_{U_i}\in K_i\}$   

\end{enumerate}

$q$ is well defined because for each component, we are taking the union of $\vert N\vert$ many pairwise compatible conditions in forcings which are at least $\alpha_n^{+\omega+2}$-closed and we have $\vert N\vert = \vert X\vert = \vert X\cap \kappa_n \vert^{+\omega+1}= \alpha_n^{+\omega+1}.$
(also note that since $N\in \bar{W}$ and $\mathbb{C}_n$ is $\kappa_n$-distributive, the condition $q$ is defined in $\bar{W}$). \end{proof}

Now we go back to $\bar{W}.$ Note that $q$ forces that the order type of $X$ is a regular cardinal, because $\alpha_n= X\cap \kappa_n,$ the order type of $X$ is 
$(X\cap \kappa_n)^{+\omega+1}$ and no cardinals between $\alpha_n$ and $\alpha_n^{+\omega+2}$ are collapsed. 
Also note that for every $\alpha\in S^*\cap X,$ there is a trivial extension of $p$ in $\mathbb{P}_n^*$ that forces $\alpha\in \dot{S^{\prime}};$ by the elementarity of $N$ we can find such an extension in $N,$ thus $q$ extends it, hence $q\force \alpha\in S^{\prime}.$ Since $X\cap S^*$ is stationary in $\sup(X),$ we have 
$q\force X\cap S^{\prime} \textrm{ is stationary in }\sup(X)$   
So we complete the proof of the theorem if we can show that $q$ forces that the subalgebra of $\dot{A}$ generated by $X$ has the same order type as $X$ and $X$ is cofinal in it.

\begin{lemma} $q$ forces that if $\dot{B}$ is the subalgebra of $\dot{A}$ generated by $X,$ then $\dot{B}$ has order type $\vert X\vert $ and $X$ is cofinal in it. 
 \end{lemma}

\begin{proof} First we show that in $\bar{W}_n[G_{\mathbb{Z}}]$ the subalgebra $B^*$ of $A^*$ generated by $\mu\times X$  has order type $\vert X\vert $ and $X$ is cofinal in it. Fix a $\mathbb{Z}$-name $\dot{h}$ for an order-preserving map from $A^*$ onto $\lambda^+.$ To show that $X$ is cofinal in $B^*$ observe that, since $\mathbb{Z}$ is $\kappa_n$-c.c., we have that, for every term 
$(\rho, \vec{\eta})\in A^*$ %(i.e. $\eta$ is a finite sequence of ordinals in $X$), 
there is a set $X_{\rho, \vec{\eta}}$ of size less than $\kappa_n$ in $\bar{W}_n$ such that $\force_{\mathbb{Z}}\ \dot{h}((\rho, \vec{\eta}))\in X_{\rho, \vec{\eta}}.$
Therefore, $\dot{h}(\rho, \vec{\eta})$ is bounded below $\lambda^+.$ 
If $\vec{\eta}$ is a finite sequence of elements of $X,$ then by elementarity of $N$ wa can assume that the bound $\beta$ is in $N,$ thus $\beta\in N\cap \lambda^+=X.$
A similar argument shows that the order type of $B^*$ is $\vert X\vert.$
Also, using again the fact that $\mathbb{Z}$ is $\kappa_n$-c.c., we can prove that $B^*$ exists in $\bar{W}_n[G_X]$ where $G_X$ the generic determined by $G_{\mathbb{Z}}$ for the product 
$$\mathbb{Z}_X:= \prod_{i<n} \Coll(\alpha_i^{+\omega+2}, <\kappa_i)\times \prod_{i<n} \Coll(\kappa_{i-1}^{++}, <\alpha_i)\times \Coll(\kappa_{n-1}^{++}, <X\cap \kappa_n).$$ 
By the distributivity of $\mathbb{C}_n/G_{\mathbb{C}_{fin}}$ we have that 
$B^*$ is in $\bar{W}[G_X].$ Let $G_{\mathbb{P}^*}$ be a $\mathbb{P}^*$-generic filter over $\bar{W}$ such that $q\in G_{\mathbb{P}^*}$ and such that the $\mathbb{Z}_X$-generic object determined by $G_{\mathbb{P}^*}$ is precisely $G_X.$ Let $B$ be the interpretation of $\dot{B}$ by $G_{\mathbb{P}^*},$ and for every $\rho<\mu,$ let $t_\rho$ be the interpretation of $\dot{t}_\rho$ by $G_{\mathbb{P}^*}.$ 
The proof of the lemma is complete if we show that in $\bar{W}[G_{\mathbb{P}^*}],$ the following map $\pi$ is an isomorphism between $B$ and $B^*.$ $\pi$ is defined as follows: given $\vec{\eta}\in X$ and 
$t_\rho,$ we map $t_\rho(\vec{\eta})\in B$ to $(\rho, \vec{\eta}).$ We need to show that for two terms $t_\rho(\vec{\eta})$ and $t_\chi(\vec{\zeta})$ in $B,$ we have $t_\rho( \vec{\eta}) = t_\chi(\vec{\zeta})$ (resp. $t_\rho(\vec{\eta})<t_\chi(\vec{\zeta})$) if and only if $(\rho, \vec{\eta}) \sim^* (\chi, \vec{\zeta})$ (resp. $(\rho, \vec{\eta})<^*(\chi, \vec{\zeta})$). 

So, suppose that $t_\rho( \vec{\eta}) = t_\chi(\vec{\zeta})$ (resp. $t_\rho(\vec{\eta})<t_\chi(\vec{\zeta})$), then by Lemma~\ref{lemma: deciding statements within P double star}, there is a condition in $\mathbb{P}^{**}$ that forces this statement. This means that there exists a trivial extension $r$ of $p$ and a stem $t\in G_X$ such that if $r^*$ is the condition obtained from $r$ by replacing the stem of $r$ with $t,$ then $r^*$ forces $t_\rho( \vec{\eta}) = t_\chi(\vec{\zeta})$ (resp. $t_\rho(\vec{\eta})<t_\chi(\vec{\zeta})$). By elementarity, we can assume that $r\in N,$ thus $q\leq r.$ It follows that $r\in G_{\mathbb{P}^*}$ and also $r^*\in G_{\mathbb{P}^*}.$ Therefore, $(\rho, \vec{\eta}) \sim^* (\chi, \vec{\zeta})$ (resp. $(\rho, \vec{\eta})<^*(\chi, \vec{\zeta})$). 

%Conversely, suppose that $(\rho, \vec{\eta}) \sim^* (\chi, \vec{\zeta})$ (resp. $(\rho, \vec{\eta})<^*(\chi, \vec{\zeta})$), then there exists $r\in \mathbb{P}^{**}$ that forces $(\rho, \vec{\eta}) \sim^* (\chi, \vec{\zeta})$ (resp. $(\rho, \vec{\eta})<^*(\chi, \vec{\zeta})$). By elementarity, we can assume that $r\in N,$ thus $q\leq r.$ Hence $q$ forces that $t_\rho(\vec{\eta}) = t_\chi(\vec{\zeta})$ (resp. $t_\rho(\vec{\eta})< t_\chi(\vec{\zeta})$) as required. \end{proof}

Conversely, suppose that $(\rho, \vec{\eta}) \sim^* (\chi, \vec{\zeta})$ (resp. $(\rho, \vec{\eta})<^*(\chi, \vec{\zeta})$), then there exists a trivial extension $r$ of $p$ and a stem $t\in G_X$ such that if $r^*$ is the condition obtained from $r$ by replacing the stem of $r$ with $t,$ then $r^*$ forces $t_\rho(\vec{\eta}) \sim^* t_\chi(\vec{\zeta})$ (resp. $t_\rho(\vec{\eta})<t_\chi(\vec{\zeta})$). By elementarity of $N,$ we can assume that $r\in N,$ hence $q$ extends it. It follows that $r\in G_{\mathbb{P}^*}$ and also $r^*\in G_{\mathbb{P}^*}.$ Therefore, 
$t_\rho(\vec{\eta}) = t_\chi(\vec{\zeta})$ (resp. $t_\rho(\vec{\eta})< t_\chi(\vec{\zeta})$) as required.  

The same argument shows that the algebra operators are preserved by the bijection.
\end{proof}

That completes the proof that the Delta reflection holds at $\lambda^+$ (hence at $\aleph_{\omega^2+1}$) in the model $\bar{W}^{\mathbb{P}}.$

\section{Open questions}
We conclude this paper with some open questions: 

\begin{question}\label{que: indestructible delta}
Can the Delta reflection principle be indestructible under any $\aleph_{\omega^2+1}$-directed closed forcing?
\end{question}
A positive answer for this question will entail a much simpler proof for the main theorem of this paper, Theorem~\ref{thm: main theorem}. A problem with similar taste is:

In \cite{MagidorShelah}, Magidor and Shelah introduced also a global Delta reflection, denoted $\Delta_\kappa.$ $\Delta_\kappa$ corresponds to the property $\Delta_{\kappa, \lambda}$ for every regular 
$\lambda\geq \kappa.$ They show that it is consistent, relative to the existence of $\omega$ many supercompact cardinals, that the successor of the first fixed point of the aleph function satisfies this principle.
 
\begin{question}
Is it consistent (relative to the existence of large cardinals) that both  $\square(\lambda^+)$ and $\Delta_{\lambda^+}$ hold where $\lambda$ is the first fixed point of the aleph function? 
\end{question}

\section{Acknowledgments}
We would like to thank Menachem Magidor for many helpful discussions. We also thank the referee for carefully reading this paper and for his comments and corrections. This research was supported by the \emph{European Commission} under a \emph{Marie Curie Intra-European Fellowship} through the project $\#\ 624381$ (acronym LAPSCA).

\bibliographystyle{elsarticle-num}

\begin{thebibliography}{10}
\expandafter\ifx\csname url\endcsname\relax
  \def\url#1{\texttt{#1}}\fi
\expandafter\ifx\csname urlprefix\endcsname\relax\def\urlprefix{URL }\fi
\expandafter\ifx\csname href\endcsname\relax
  \def\href#1#2{#2} \def\path#1{#1}\fi
\bibitem{Magidor-Cummings-Foreman-Squares}
J.~Cummings, M.~Foreman, M.~Magidor,
  \href{http://dx.doi.org/10.1142/S021906130100003X}{Squares, scales and
  stationary reflection}, J. Math. Log. 1~(1) (2001) 35--98.
\newblock \href {http://dx.doi.org/10.1142/S021906130100003X}
  {\path{doi:10.1142/S021906130100003X}}.
\newline\urlprefix\url{http://dx.doi.org/10.1142/S021906130100003X}


\bibitem{Eklof}
P.~Eklof, On the existence of $\kappa$ free abelian groups, in: Proceedings of
  the American Mathematical Society, Vol.~47, 1975, pp. 65--72.
\bibitem{FontanellaMagidor}
L.~Fontanella, M.~Magidor, Reflection of stationary sets and the tree property
  at the successor of a singular cardinal, to appear in Journal of Symbolic
  Logic.

\bibitem{VopenkaHajek}
P.~H{\'a}jek, P.~Vop{\v{e}}nka, The theory of semisets, Academia (Publishing
  House of the Czechoslovak Academy of Sciences), Prague, 1972.

\bibitem{Jensen}
R.~Jensen, The fine structure of the constructible hierarchy, Annals of
  Mathematical Logic 4 (1972) 229--308.

\bibitem{Laver}
R.~Laver, Making the supercompactness of {$\kappa $} indestructible under
  {$\kappa $}-directed closed forcing, Israel J. Math. 29~(4) (1978) 385--388.


\bibitem{MagidorShelah}
M.~Magidor, S.~Shelah, When does almost free implies free, Journal of the
  American Mathematical Society 7~(4) (1994) 769 -- 830.

\bibitem{MilnerShelah}
E.~C. Milner, S.~Shelah, Some theorems on transversals, in: Infinite and finite
  sets (Colloq., Keszthely, 1973; dedicated to P. Erd\H{o}s on his 60th
  birthday), Vol. 10 (III) of Colloq. Math. Soc. Janos Bolyai, North Holland,
  Amsterdam, 1975, pp. 1115--1126.

\bibitem{Shelah}
S.~Shelah, A compactness theorem for singular cardinals, free algebras,
  whitehead problem and transversals, Israel Journal of Mathematics 21~(4)
  (1975) 319--349.

\bibitem{ShelahStanley}
S.~Shelah, L.~Stanley, Weakly compact cardinals and non special aronszajn
  trees, in: Proc. AMS, Vol. 104, 1988, pp. 887--897.

\bibitem{Solovay}
R.~M. Solovay, Strongly compact cardinals and the gch, in: L.~Henkin (Ed.),
  Proceedings of the Tarski symposium, Vol.~25 of Colloq. Math. Soc. Janos
  Bolyai, Providence, Rhode Island, 1974, pp. 365--372.

\bibitem{Todorcevic}
S.~Todor{\v{c}}evi{\'c},
  \href{http://dx.doi.org/10.1007/BF02392561}{Partitioning pairs of countable
  ordinals}, Acta Math. 159~(3-4) (1987) 261--294.
\newblock \href {http://dx.doi.org/10.1007/BF02392561}
  {\path{doi:10.1007/BF02392561}}.
\newline\urlprefix\url{http://dx.doi.org/10.1007/BF02392561}

\bibitem{Rinot}
A.~Rinot, \href{http://dx.doi.org/10.1017/bsl.2014.24}{Chain conditions of
  products, and weakly compact cardinals}, Bull. Symb. Log. 20~(3) (2014)
  293--314.
\newblock \href {http://dx.doi.org/10.1017/bsl.2014.24}
  {\path{doi:10.1017/bsl.2014.24}}.
\newline\urlprefix\url{http://dx.doi.org/10.1017/bsl.2014.24}

\bibitem{Boban}
B.~Velickovic, \href{http://dx.doi.org/10.2307/2273941}{Jensen's principles and
  the novak number of partially ordered sets}, J. Symb. Log. 51~(1) (1986)
  47--58.
\newblock \href {http://dx.doi.org/10.2307/2273941}
  {\path{doi:10.2307/2273941}}.
\newline\urlprefix\url{http://dx.doi.org/10.2307/2273941}

\end{thebibliography}


%\begin{thebibliography}{00}


%\bibitem{}

%\end{thebibliography}

\end{document}